\documentclass[10pt,article]{amsart}
\usepackage{latexsym,amssymb,amsmath,eufrak}

\newtheorem {theo}{Theorem}[section]
\newtheorem {prop}[theo]{Proposition}
\newtheorem {lem}[theo]{Lemma}

\newtheorem {cor}[theo]{Corollary}
\newtheorem {defi}[theo]{Definition}

\newtheorem {rmk}[theo]{Remark}
\newcommand{\eqnsection}{
   \renewcommand{\theequation}{\thesection.\arabic{equation}}
   \makeatletter
   \csname @addtoreset\endcsname{equation}{section}
   \makeatother}

\newcommand{\be}{{\begin{equation}}}
\newcommand{\ee}{{\end{equation}}}
\def \bt{\begin{theorem}}
\def \et{\end{theorem}}

\def \al{\alpha}

\def \de{\delta}
\def \De{\Delta}

\def \tr{\nabla}

\def \wh{\widehat}
\def \wt{\widetilde}

\def \R{{\Bbb{R}}}

\def \C{{\bf C}}

\def \bX{{\bf X}}
\def \bI{{\bf I}}
\def \bB{{\bf B}}
\def \bW{{\bf W}}
\def \bY{{\bf Y}}

\def \bx{{\bf x}}

\def\Bbb{\mathbb}

\newcommand{\reals}{{\Bbb{R}}}

\newcommand{\beq}[1]{\begin{equation}\label{#1}}
\newcommand{\eeq}{\end{equation}}

\newcommand{\E}{{\Bbb E}}

\newcommand{\beqn}[1]{\begin{eqnarray}\label{#1}}
\newcommand{\eeqn}{\end{eqnarray}}
\newcommand{\oo}{\overline}
\newcommand{\uu}{\underline}

\def\squarebox#1{\hbox to #1{\hfill\vbox to #1{\vfill}}}

\newcommand{\beaa}{\begin{eqnarray*}}
\newcommand{\eeaa}{\end{eqnarray*}}

\def\P{\mathbb P}
\def\PC{{\mathcal C_\star}}
\def\E{\mathbb E}

\def\E{\mathbb E}
\def\C{\mathbb C}
\def\e{\epsilon}

\def \DD{{\mathcal D}}
\def \FF{{\mathcal F}}

\def \KK{{\mathcal K}}
\def \LL{{\mathcal L}}

\def \OO{{\mathcal O}}

\def\N{{\mathbb N}}

\def\tr{{\mbox{tr}}}

\def\oo{\overline}
\def\tr{{\mbox{tr}}}

\def\uu{\underline}

\def\bY{{\bf Y}}
\def\bA{{{\bf A}}}
\def\bB{{{\bf B}}}

\def\bx{{\boldsymbol x}}
\def\bz{{\boldsymbol z}}

\def\bsigma{{\boldsymbol \sigma}}

\def\bG{{\boldsymbol G}}
\def\bD{{\boldsymbol D}}
\def\bB{{\boldsymbol B}}
\def\bW{{\boldsymbol W}}
\def\bt{{\bf t}}

\def\a{\alpha}

\def\d{\delta}
\def\e{\epsilon}
\def\g{\gamma}

\def\s{\sigma}

\def\arg{\mbox{arg}}

\def\D{\Delta}

\def\ra{\rightarrow}

\def\Ca{{\mathcal C}}

\def\Ea{{\mathcal E}}

\def\KK{{\mathcal K}}
\def\rBL{{\bf BL}}

\def\mun{{\hat\mu}}

\def\lbc{\lbrace}

\def\part{\partial}

\def\ts{\times}
\def\Pa{{\mathcal{P}}}

\def\R{{\mathbb R}}

\def\sigm{ \sigma (\frac{i}{N}, \frac{j}{N})}

\def\Eaa{\Ea_\alpha}
\def\Eaae{\Ea_{\alpha,\varepsilon}}
\def\oEaae{\oo{\Ea}_{\alpha,\varepsilon}}
\def\Ball{\mathbb B}
\def\Db{\mathbb K}
\def\Vb{\mathbb V}
\def\Ub{\mathbb U}
\def\Mb{\mathbb M}
\def\Fb{\mathbb F}
\def\Ib{\mathbb I}
\def\Sb{\mathbb S}
\def\Hb{\mathbb H}
\def\Ab{\mathbb A}
\def\Abs{\mathbb D}
\eqnsection
\begin{document}

\title[Heavy tailed band matrices]{Spectral measure of heavy tailed 
band and covariance random matrices}
\author[Serban Belinschi,\,\,
Amir Dembo\,\, Alice Guionnet]{
Serban Belinschi$^\star$\,\,
Amir Dembo$^\dagger$\,\, Alice Guionnet$^\ddagger$}

\date{July 10, 2008; Revised: January 27, 2009.
\newline\indent 
$^\star$ University of Saskatchewan and Institute of Mathematics of the Romanian Academy, 
106 Wiggins Road, Saskatoon, Saskatchewan S7N 5E6, Canada.
E-mail: belinschi@math.usask.ca.
Supported in part by a Discovery grant from the Natural Sciences and
Engineering Research Council of Canada and a University of Saskatchewan
start-up grant.
\newline\indent 
$^\dagger$ Stanford University, Stanford, CA 94305, USA.
E-mail: amir@math.stanford.edu. 
Research partially supported by NSF grant \#DMS-0806211.
\newline\indent
$^\ddagger$ Ecole
Normale Sup\'erieure de Lyon, Unit\'e de Math\'ematiques pures et
appliqu\'ees, UMR 5669,46 All\'ee d'Italie, 69364 Lyon Cedex 07,
France. E-mail: aguionne@umpa.ens-lyon.fr and Miller institute for
Basic Research in Science, University of California Berkeley.
\newline
\newline 
{\bf AMS (2000) Subject Classification:}
{Primary: 15A52, Secondary: 30D40, 60E07, 60H25}
\newline
{\bf Keywords:} Random matrices, Levy matrices, 
$\alpha$-stable variables, Spectral measures,
Wishart matrices, Pastur-Marchenko law.
}

\maketitle

\begin{abstract}
We study the asymptotic behavior of the appropriately
scaled and possibly perturbed 
spectral measure $\mun$  of large
random real symmetric matrices with  heavy tailed 
entries.
Specifically, consider the $N\ts N$  symmetric matrix $\bY_N^\sigma$
whose $(i,j)$  
entry is  $\sigma(\frac{i}{N}, \frac{j}{N})
x_{ij}$ where $(x_{ij}, 1\le i\le j< \infty)
$ is  an infinite array of i.i.d real variables with common
distribution 
in the domain of attraction of an $\alpha$-stable law,
$\alpha\in (0,2)$,
and $\sigma$ is a deterministic 
function.
For random diagonal $\bD_N$ independent of $\bY_N^\sigma$
and with appropriate rescaling $a_N$,
we prove that $\mun_{a_N^{-1} \bY_N^\sigma + \bD_N}$ 
converges in mean towards a limiting probability measure 
which we characterize. As a special case, we derive 
and analyze the almost sure 
limiting spectral density for empirical covariance 
matrices with heavy tailed entries.
\end{abstract}

\section{Introduction}
We study the asymptotic behavior of the spectral measure of large
band random real symmetric matrices with independent (apart from 
symmetry) heavy tailed entries. Specifically, with 
$(x_{ij}, 1\le i\le j< \infty)$ an infinite array of i.i.d real variables, 
let $\bX_N$ denote the $N\ts N$ symmetric matrix given by
$$ X_N(i,j)=x_{ij}\,  \mbox{ if }\, i\le j,\, x_{ji}\, \mbox{ otherwise.}$$
Fixing $\sigma: [0,1]\ts [0,1]\ra \R$, a 
(uniformly over $1/N$-lattice grids) 
square integrable 
measurable function such that $\sigma(x,y)=\sigma(y,x)$, we 
denote by $\bY_N^\sigma$ the $N\ts N$ symmetric matrix with entries
$Y_N^\sigma(i,j)= \sigm x_{ij}$.
These matrices are sometime called 
``band matrices'' after the choice of 
$\sigma(x,y)={\bf 1}_{|x-y| \leq b}$ for some $0<b<1$
(c.f. Remark \ref{rmkbnd}). Another important special case, 
$\sigma(x,y)={\bf 1}_{(x-1/2)(1/2-y)>0}$ yields 
the spectral measure of empirical covariance matrices
$\bX_N \bX_N^t$ (as shown in Section \ref{sec:wis-conv}).

For i.i.d. entries
$(x_{ij}, 1\le i\le j\le N)$ of finite 
second moment,
it was proved by Berezin 
that the spectral measure of 
$\bA^\s_N := N^{-1/2} \bY_N^\sigma$ converges
almost surely weakly
(see a rigorous proof in \cite{KKPS}). 
More precisely, for any $z\in\C\backslash\R$ the matrices 
$\bG_N(z):=(z\bI_N-\bA_N^\s)^{-1}$ are such that 
for any bounded continuous function $\phi$,  
$$ 
\lim_{N\ra\infty}\frac{1}{N}\sum_{i=1}^N \phi(\frac{i}{N}) G_N(z)_{ii}=
\int_0^1\phi(u) K^\s_u(z) du\, \quad {\rm a.s.}
$$
with $K^\s_x(z)$ the unique solution of
$K^\s_x(z)=(z-\int_0^1 |\s(x,v)|^2 K^\s_v (z) dv)^{-1}$ 
such that $z\mapsto \int_0^1\phi(u) K^\s_u (z) du$
is analytic in $\C\backslash\R$. In particular, taking 
constant $\phi(\cdot)$ 
we have the almost sure convergence of 
the spectral measure of 
$\bA_N^\sigma$ 
to the probability measure $\mu^\sigma_2$ 
whose Cauchy-Stieltjes transform is 
\begin{equation}\label{wishartclass}
G^\sigma_2(z)=\int\frac{1}{z-\lambda}d\mu^\sigma_2(\lambda)
=\int_0^1 K^\s_v (z) dv \,.
\end{equation}

We consider here the case of heavy tailed entries, where the common
distribution of the absolute values of the $x_{ij}$'s is in the
domain of attraction of an $\alpha$-stable law, for $\alpha\in
]0,2[$. 
That is, 
there exists a slowly varying function $L(\cdot)$ 
such that for any $u>0$,
\begin{equation}
\P(|x_{ij}|\geq u)= L(u) u^{-\alpha} \,. \label{stabledomain}
\end{equation}
The normalizing constants 
\begin{equation}
a_N: = \inf\{u: \, \P[ |x_{ij}| \geq u]\le \frac{1}{N}\,\}\,, 
\label{normalisation}
\end{equation}
are then such that $a_N = L_0(N) N^{1/\alpha}$ for some 
(other) slowly varying function $L_0(\cdot)$.

Hereafter, let $\bA_N^\sigma$ denote the normalized 
matrix $\bA_N^\sigma:= a_N^{-1} \bY_N^\sigma$ having eigenvalues
$(\lambda_1$, $\cdots$, $\lambda_N)$ and the corresponding 
spectral measure
$\mun_{\bA_N^\sigma}:=\frac{1}{N}\sum_{i=1}^N \delta_{\lambda_i}$
(and when the choice of 
$\sigma(\cdot)$ is clear we also use the notations $\bY_N$ and $\bA_N$ for
$\bY_N^\sigma$ and $\bA_N^\sigma$, respectively). 
Predictions about 
the limiting spectral measure in case $\sigma(\cdot,\cdot) \equiv 1$
(the heavy tail analog of Wigner's theorem) have been made
in \cite{CB} and rigorously verified in \cite{BAG6} (c.f. 
\cite[Section 8]{BAG6}). We follow here the approach of 
\cite{BAG6},
which consists of
proving the convergence of 
the resolvent, i.e. of the mean of the 
Cauchy-Stieltjes transform 
of the spectral measure, outside of the real line, by proving 
tightness and characterizing uniquely the possible limit points.
In the latter task, for each $\alpha \in (0,2)$
the limiting spectral measure 
of $\bA_N^\s$
is characterized in terms of the entire functions
\begin{eqnarray}\label{defg}
g_{\alpha}(y)&: =&
\int_0^\infty t^{\frac{\alpha}{2}-1} e^{-t} \exp\{-
t^{\frac{\alpha}{2}} y\} dt \,, \\
\label{deffal}
h_\alpha(y) &:=& \int_0^\infty \! e^{-t}
\exp\{-t^{\frac{\alpha}{2}} y\} dt 
= 1- \frac{\alpha}{2} y g_\alpha (y) \,.
\end{eqnarray}
We define 
for any $\alpha\in (0,2)$ the usual branch of the power function 
$x\mapsto x^{\alpha}$, which is the analytic function on 
$\C\backslash \R^-$
such that $(i)^{\alpha}= e^{i\frac{\pi\alpha}{2}}$. This amounts to
choosing $x^\alpha= r^\alpha e^{i \alpha \theta}$
when $x=r e^{i\theta}$ with $\theta\in ]-\pi,\pi[$. We also 
adopt throughout the notation $x^{-\alpha}$ for $(x^{-1})^\alpha$.
With these notations in place, recall \cite[Theorem 1.4]{BAG6} 
that in case $\sigma(\cdot,\cdot)\equiv 1$, the limiting spectral 
measure $\mu_\alpha$ 
for Wigner matrices with entries in the domain 
of attraction of an $\alpha$-stable law 
has for $z \in \C^+ = \{ z \in \C : \Im(z)>0 \}$, 
the Cauchy-Stieltjes transform 
\begin{equation}\label{new-amir-eq3}
G_\alpha(z) 
:= \int \frac{1}{z-x} d\mu_\alpha(x) = \frac{1}{z} h_\alpha(Y(z)) \,,
\end{equation}
where $Y(z)$ is the unique analytic on $\C^+$ solution of 
\begin{equation}\label{eq:BAG6}
z^\a Y(z) = C_\alpha g_\alpha(Y(z))
\end{equation}
tending to zero at infinity, and 
$C_\alpha:=
i^{\alpha} \Gamma(1-\frac{\alpha}{2})/\Gamma(\frac{\alpha}{2})$. 
In \cite[Theorem 1.6]{BAG6}  it is further shown that $\mu_\alpha$
has a smooth symmetric density $\rho_\alpha$ outside a compact 
set of capacity zero, and that 
$t^{\a+1} \rho_\alpha(t) \to \a/2$ as $t \to \infty$. 

In addition to considering the more general case of band matrices, 
we devote some effort to the analysis of the limiting 
Cauchy-Stieltjes transform as $\Im(z) \to 0$ and its consequences 
on existence and regularity of the limiting density.
For example, as a by product of our analysis  
we prove the following about $\mu_\alpha$
of \cite{BAG6}, showing in particular that it has a 
uniformly bounded density.
\begin{prop}\label{gammaone} 
The unique analytic on $\C^+$ 
solution $Y(z)$ of 
\eqref{eq:BAG6} tending to zero at
infinity takes values in the set 
${\mathcal K}_\alpha:=
\{Re^{i\theta}:
|\theta|\le \frac{\alpha\pi}{2}, \, R \geq 0\}$ on which  
$g_\alpha(\cdot)$ is uniformly bounded. Its
continuous extension to $\R \setminus \{0\}$ 
is analytic except possibly at the finite set 
$\DD_\alpha = \{0 , \pm t: t^\alpha = C_\alpha g_\alpha'(y) >0,
y \in {\mathcal K}_\alpha, g_\alpha(y)=y g_\alpha'(y) \}$. 
Further, the symmetric uniformly bounded density of $\mu_\alpha$ is 
\begin{equation}\label{rhoalpha}
\rho_\alpha (t) = -\frac{1}{\pi t} 
\Im\big(h_\alpha(Y(t)) \big) 
=\frac{\alpha |t|^{\alpha-1}}{2 |C_\a| \pi} \Im\big(i^{-\a} Y(|t|)^2\big) 
\,,
\end{equation}
continuous at $t \ne 0$, real-analytic outside $\DD_\alpha$ and 
non-vanishing on any open interval.
\end{prop}

\begin{rmk}\label{rem-alpha-cont}
It is noted in \cite[Remark 1.5]{BAG6} that 
$\a \mapsto \mu_\a$ is continuous on $(0,2)$ 
with respect to weak convergence of probability 
measures. We further show in Lemma \ref{alice-cont}
that as $\a \to 2$ the measures $\mu_\a$ converge
to the semi-circle law $\mu_2$.
\end{rmk}

Let $\PC$ denote the set of piecewise constant functions
$\sigma(x,y)$ such that for some finite $q$, 
some $0=b_0<b_1<\cdots<b_{q}=1$ and a $q \ts q$ 
symmetric matrix of entries $\{\sigma_{rs}, 1 \le r,s\le q\}$,
\begin{equation}\label{eq:sPCi}
\sigma(x,y)=\sigma_{rs}\qquad \mbox{\rm for all} \quad
(x,y) \in (b_{r-1},b_r] \ts (b_{s-1},b_s] \; .
\end{equation}
Our next result
provides the weak convergence of the spectral measures for
$\bA_N^\sigma$ and characterizes 
the Cauchy-Stieltjes transform of their limit, in case 
$\sigma \in \PC$. 
Even for $\sigma(\cdot,\cdot) \equiv 1$ it goes beyond 
the results of \cite{BAG6} by strengthening 
the weak convergence of the expected spectral measures  
$\E[\mun_{\bA_N}]$ to the weak convergence 
of $\mun_{\bA_N}$ holding with probability one. 
A special interesting case of $\sigma$ is when 
$q=2$ and $\sigma_{rs}={\bf 1}_{|r-s|=1}$, out of 
which we get the spectral measure of the empirical 
covariance matrices $a_N^{-2} \bX_N \bX_N^t$ 
(c.f. Theorem \ref{wishart-amir}
and its proof in Section \ref{sec:wis}).

\begin{theo}\label{theo-limitpoint-uniq-amir}
Fixing $\sigma \in \PC$, let $\D_r=b_r-b_{r-1}$ for
$r=1,\ldots,q$. With probability one,
the sequence $\mun_{\bA_N^\sigma}$ 
converges weakly towards 
the non-random, symmetric probability measure $\mu^{\sigma}$. 
The limiting measure has a
continuous density
$\rho^\s$ on $\R\backslash\{0\}$ which is 
bounded off zero, 
and its Cauchy-Stieltjes transform is, for any $z\in \C^+ $, 
\begin{equation}\label{eqG} 
G_{\alpha,\sigma}(z):=\int \frac{1}{z-x}d\mu^{\sigma}(x)=
\frac{1}{z} \sum_{s=1}^q \D_s h_\alpha(Y_s(z))  \,,
\end{equation}
where $\uu{Y}(z) \equiv (Y_r(z), 1\le r\le q)$ is the unique solution of 
\begin{equation}\label{systemeqY}
z^\alpha Y_r(z)=C_\a \sum_{s=1}^q |\sigma_{rs}|^\alpha \D_s
 g_{\alpha}(Y_s(z))  \,,
\end{equation}
composed of functions that are analytic on $z\in \C^+$ and  
tend to zero as $|z| \to \infty$. Moreover, 
$z^\alpha \uu{Y} (z)$ is uniformly bounded on $\C^+$,
both $G_{\alpha,\sigma}(z)$ and
$\uu{Y}(z) \in ({\mathcal K}_\alpha)^q$
have continuous, algebraic extensions to 
$\reals \setminus \{0\}$, 
and for some $R=R(\sigma)$ finite the mapping $\uu{Y}(z)$
extends analytically through the subset $(R,\infty)$  
%
where $\rho^\s(t)=-\frac{1}{\pi t}
\sum_{s=1}^q \D_s \Im\big(h_\alpha(Y_s(t))\big)$
is real-analytic. 
Finally, the map $z\mapsto \uu{Y}(z)$ is injective 
whenever $\sigma\not\equiv 0$.
\end{theo}

\begin{rmk} 
The measure 
$\mu^\s$ may have an atom at zero
when $q > 1$. Indeed, Theorem \ref{wishart-amir}
provides one such example in case $q=2$.
\end{rmk}

\begin{rmk}
While we do not pursue it here, 
similarly to \cite[Section 9]{BAG6}, one can apply
the moment method developed by Zakharevich \cite{zakh},
to characterize $\mu^\s$ as the weak limit
$B \to \infty$ of the limiting spectral measures
for appropriately truncated matrices $\bA_N^{\s,B}$.
As done in Lemma \ref{alice-cont} for $\s \equiv 1$,
we expect this to yield the continuity of $\mu^\s$ 
with respect to $\a \to 2$, for each fixed 
$\s \in \PC$, i.e. to connect the
limiting measures of Theorem \ref{theo-limitpoint-uniq-amir}
to $\mu^\s_2$ of (\ref{wishartclass}).
\end{rmk}

Let $L^2_\star([0,1]^2)$ denote the space of 
equivalence classes with respect to the semi-norm 
$$
\|f\|_\star := \limsup_{n \to \infty} 
\| f(n^{-1} \lceil nx \rceil,n^{-1} \lceil ny\rceil) \|_2
\,,
$$
on the space of functions on $[0,1]^2$ for which 
$\|\cdot\|_\star$ is finite.
For each measurable $f:[0,1]^2 \mapsto \R$ let
$\|f\| := \| \int_0^1 |f(x,v)| dv \|_\infty$ denote
the associated operator norm, where   
$\|\cdot\|_\infty$ denotes hereafter the usual 
(essential-sup) norm of $L^\infty((0,1])$. 
We consider the subset ${\mathcal F}_\alpha$ of those
symmetric 
measurable 
functions $\s \in L^2_\star ([0,1]^2)$ 
with $\| \, |\s|^\alpha \|$ finite  
which are each the $L^2_\star$-limit of some sequence $\s_p\in\PC$ such that
\begin{equation}\label{domaine}
\lim_{p\ra\infty}\| \, |\s_p|^\alpha - |\s|^\alpha \|=0\,.
\end{equation}
In fact, to verify that $\s \in \FF_\alpha$ it suffices to 
check that $\| \, |\s|^\alpha \|$ is finite and  
find $L^2_\star$-approximation of $\s(\cdot,\cdot)$ by bounded continuous 
symmetric functions $\s_p(\cdot,\cdot)$ for which 
\eqref{domaine} holds. 
Obviously ${\mathcal F}_\alpha$ contains all 
bounded continuous symmetric functions on $[0,1]^2$ 
(but for example 
$\sigma(x,y)=1/\sqrt{x+y} \in L^2_\star([0,1]^2)$ is not in $\FF_\alpha$).

\begin{rmk}
Things are a bit simpler if in the definition of the matrix
$\bY_N^\sigma$ one replaces the sample $\sigm$ by the 
average of $\sigma(\cdot,\cdot)$ with respect to 
Lebesgue measure on $(\frac{i-1}{N},\frac{i}{N}] \times (\frac{j-1}{N},
\frac{j}{N}]$, for then we can replace throughout this paper
the semi-norm $\|\cdot\|_\star$ and the  
space $L^2_\star([0,1]^2)$ by the usual $L^2$-norm and space.
\end{rmk}

We further say that $\s \in \FF_\alpha$ is equivalent to 
$\wt{\s} \in \PC$ if for the relevant finite partition
$0=b_0<b_1<\cdots<b_{q}=1$ we have for any $1 \le r,s\le q$ that 
$$
\int_{b_{s-1}}^{b_s} |\s(x,v)|^\alpha dv = |\wt{\s}_{rs}|^\alpha 
\qquad \mbox{\rm for all} \quad x \in (b_{r-1},b_r] \,.
$$

Extending Theorem \ref{theo-limitpoint-uniq-amir} 
we next characterize the Cauchy-Stieltjes transform of $\mu^{\sigma}$
for any $\s \in \FF_\alpha$.

\begin{theo}\label{weakening}
Given $\sigma\in {\mathcal F}_\alpha$, 
the sequence $\E[\mun_{\bA_N^\sigma}]$ converges 
weakly towards the symmetric 
probability measure $\mu^{\sigma}$ 
such that
for some $R=R(\s)$ finite, 
\begin{equation}\label{eqGlim}
\int \frac{1}{z-x} d\mu^\sigma(x) =
\frac{1}{z} \int_0^1 h_\alpha(Y^\s_v(z)) dv 
\end{equation}
and $Y^\s$ is the unique analytic mapping
$Y^\sigma:\C^+ \mapsto L^\infty((0,1];{\mathcal K}_\alpha)$
such that if $|z| \ge R$ then 
for almost every $x \in (0,1]$ 
\begin{equation}\label{systeqlim}
z^\alpha Y^{\sigma}_x(z)= C_\alpha \int_0^1 |\sigma(x,v)|^\alpha
g_\alpha(Y_v^\sigma(z)) dv \,. 
\end{equation}
The measure $\mu^\s$ has a density $\rho^\s$ 
on $\R\backslash\{0\}$ which is bounded off zero 
and such that 
$t^{\a+1}\rho^\s(t) \to \frac{\alpha}{2} 
\int |\s(x,v)|^\a dx dv$ as $t \to \infty$.

Further, 
if  $\s \in \FF_\alpha$ is equivalent to 
$\wt{\s} \in \PC$ then $\mu^\s=\mu^{\wt{\s}}$.
\end{theo}

\begin{rmk} A similar invariance applies 
in case of entries with bounded variance, where the kernel $K^\s_x(z)$
that characterizes the limit law in \eqref{wishartclass}
is the same across each equivalence class of $\FF_2$.
Also note that for $\alpha=2$ we have
$C_2=-1$ and $g_2(y)=h_2(y)=1/(y+1)$ is well defined when 
$\Re(y)>-1$. Plugging the latter expressions into \eqref{eqGlim}
and \eqref{systeqlim} indeed coincide with \eqref{wishartclass}
upon setting $z K^\s_x(z)=g_2(Y^\s_x(z))=1/(1+Y^\s_x(z))$,
whereas \eqref{new-amir-eq3} and \eqref{eq:BAG6}
result for $\alpha=2$ 
with $Y(z)=-\frac{1}{z} G_2(z)$ and the 
Cauchy-Stieltjes transform 
$G_2(z)=(z-\sqrt{z^2-4})/2$ of the semi-circle law $\mu_2$
(upon properly choosing the branch of the square root).
\end{rmk}

\begin{rmk}\label{rmkbnd}
The equivalence between $\s \in \FF_\alpha$ 
and $\wt{\s} \in \PC$ is often quite useful. For example, if  
$\varphi:[-1,1] \to \R$ is any even, periodic function of 
period one and finitely many jump discontinuities
then $\s(x,y)=\varphi(x-y) \in \FF_\alpha$ 
and is equivalent to the 
constant $\wt{\s}=[\int_0^1 |\varphi(v)|^\alpha dv]^{1/\alpha}$.
Consequently, in this case $\mu^\s$ equals $\mu_\alpha(\wt{\s} \cdot)$
of \cite{BAG6} and hence has the symmetric, uniformly bounded,
continuous off zero, density 
$\wt{\s}^{-1} \rho_\alpha(t/\wt{\s})$ 
with respect to Lebesgue measure on $\R$. 
\end{rmk}

Consider next the empirical covariance matrices
$\bW_{N,M}=a_{N+M}^{-2} \bX_{N,M} \bX_{N,M}^t$ 
where $\bX_{N,M}$ is an $N\ts M$ matrix 
with heavy tailed entries 
$x_{ij}$, $1\le i\le N$, $1\le j\le M$, the law of which 
satisfies \eqref{stabledomain} (and $\bB^t$ denotes throughout the 
transpose of the matrix $\bB$). 
Taking $N \to \infty$ and $M/N \to \gamma \in (0,1]$ 
the scaling constant $a_N$ is chosen per \eqref{normalisation}
(so from \eqref{stabledomain} we have that 
$a_{N+M}^2 \sim N^{\frac{2}{\alpha}} (1+\gamma)^{2/\alpha} L_1(N)$ 
for some slowly varying function $L_1(\cdot)$). In this setting
we show the following about the
limiting spectral measure of $\bW_{N,M}$.

\begin{theo}\label{wishart-amir}
If $N \to \infty$ and  $\frac{M}{N} \to \gamma\in (0,1]$ then 
with probability one, the spectral measures $\mun_{\bW_{N,M}}$
converge to a non-random probability measure $\mu^\gamma_\alpha$. 
The probability measure $\mu_\a^1$ is absolutely continuous 
with the density 
$$
\rho^1_\alpha(t)
=2^{1/\alpha} t^{-1/2} \rho_\alpha(2^{1/\alpha} \sqrt{t}) 
$$ 
on $(0,\infty)$. 
Fixing $\gamma \in (0,1)$
let $(Y_1(z),Y_2(z))$ denote the unique analytic functions 
of 
$z \in \C^+$ 
tending to zero at infinity, such that 
\begin{equation}\label{defywis}
z^\alpha Y_1(z)=\frac{\gamma}{1+\gamma} C_\a
 g_{\alpha}(Y_2(z))\,, \qquad  z^\alpha Y_2(z) 
=\frac{1}{1+\gamma} C_\a g_{\alpha}(Y_1(z))  \,.
\end{equation}
The functions $Y_1(z)$ and $Y_2(z)$
extend continuously to functions on $(0,\infty)$ that are 
analytic through $(R,\infty)$ for some finite 
$R=R_\alpha^\gamma$.  The probability measure $\mu_\a^\gamma$ 
then has an atom at zero of mass  $1-\gamma$
and the continuous density 
\begin{equation}\label{eq:den-wis}
\rho^\gamma_\alpha(t) 
=-\frac{1}{\pi t} \Im\big( h_\alpha(Y_1(\sqrt{t}\,) \big) \,,
\end{equation}
on $(0,\infty)$ which 
is real-analytic on $(R,\infty)$, bounded off zero,  
does not vanish in any neighborhood of zero and
such that 
$t^{1+\a/2} \rho^\gamma_\alpha(t) \to \frac{\a \gamma}{2(1+\gamma)}$
as $t \to \infty$.
\end{theo}

\begin{rmk}
Note the contrast between the non-vanishing near zero density 
$\rho_\a^\g$ and the Pastur-Marchenko law $\mu_2^\gamma$
which vanishes 
throughout 
$[0,1-\gamma]$ (c.f. \cite{Pastur}).
\end{rmk}

We also consider diagonal perturbations of heavy tailed
matrices. Namely, the limit of the 
spectral measures $\mun_{\bA_N^\s + \bD_N}$ where 
$\bD_N$ is a diagonal $N \ts N$ matrix, whose entries 
$\{D_N(k,k), 1 \le k \le N \}$ are real valued,  
independent of the random variables $(x_{ij}, 1 \leq i \le j < \infty)$
and identically distributed, of law $\mu^{\bD}$ which has
a finite second moment. In this setting we have the 
following extension of Theorem \ref{theo-limitpoint-uniq-amir}
and Theorem \ref{weakening}. 
\begin{theo}\label{weakeningD}
Let $\wh{\KK}_\alpha := \{R_0 e^{i\varphi}:
-\frac{\alpha\pi}{2} \le \varphi \le 0 ,\, R_0 \geq 0\}$. 
Given $\sigma\in {\mathcal F}_\alpha$, 
the sequence $\E[\mun_{\bA_N^\sigma+\bD_N}]$ converges 
weakly towards the probability measure $\mu^{\sigma,\bD}$ whose 
Cauchy-Stieltjes transform at $z \in \C^+$ is 
\begin{equation}\label{eqGlimD}
G^{\bD}_{\a,\s} (z) 
= \int \frac{1}{z-\lambda} 
d\mu^{\bD}(\lambda) 
\int_0^1 h_\a( (\lambda-z)^{-\frac{\alpha}{2}} 
\wh{X}^{\s}_v (z) ) dv \,,
\end{equation}
for 
some $R=R(\sigma)$ finite and
the unique analytic mapping
$\wh{X}^{\sigma} : \C^+ \mapsto L^\infty((0,1];\wh{\KK}_\alpha)$
such that if $\Im(z) \ge R(\sigma)$ then 
for almost every $x \in (0,1]$ 
\begin{equation}\label{systeqlimD}
\wh{X}^{\sigma}_x(z)= \oo{C}_\a \int_0^1 |\sigma(x,v)|^\alpha
\int (\lambda-z)^{-\frac{\alpha}{2}} 
g_{\alpha} \big( 
(\lambda-z)^{-\frac{\alpha}{2}} \wh{X}^{\sigma}_v(z)
\big) d\mu^{\bD}(\lambda) \, dv \,.
\end{equation}
If $\s \in \PC$ then $\wh{X}^{\sigma}_x (z)$ takes the same 
value $\wh{X}_r(z)$ for all $x \in (b_{r-1},b_r]$, where 
$(\wh{X}_r(z), 1\le r\le q)$ is the unique collection of 
analytic functions from $\C^+$ to $\wh{\KK}_\alpha$ such that 
\begin{equation}\label{systemeqhX}
\wh{X}_r(z) = \oo{C}_\a \sum_{s=1}^q |\sigma_{rs}|^\alpha \D_s
\int (\lambda-z)^{-\frac{\alpha}{2}} 
g_{\alpha} \big(  (\lambda-z)^{-\frac{\a}{2}} \wh{X}_s(z) \big) 
d\mu^{\bD}(\lambda) 
\end{equation}
and $|\wh{X}_r(z)| \le c (\Im(z))^{-\frac{\alpha}{2}}$ 
for some finite $c$ and all $r\in\{1,\ldots,q\}$.   
\end{theo}

\begin{rmk}
The 
substitution of $g_2(y)=h_2(y)=1/(1+y)$ in 
\eqref{systeqlimD} and \eqref{systemeqhX} leads to
the prediction 
$G^{\bD}_{2,\s}(z)=\int (\lambda-z-\wh{X}^\s_v(z))^{-1} dv d\mu^{\bD}(\lambda)$
with 
$\wh{X}^\s_x(z)=\int |\s(x,v)|^2 (\lambda-z-\wh{X}^\s_v(z))^{-1} dv 
d\mu^{\bD}(\lambda)$ which in particular for $\sigma(\cdot,\cdot)\equiv 1$
results with $\wh{X}^\s_x(z)=G^{\bD}_{2}(z)$ independent of $x$ that 
corresponds to the celebrated 
free-convolution of $\mu^{\bD}$ and $\mu_2$. 
Namely, 
$G^{\bD}_2(z)=\int (\lambda-z-G^{\bD}_2(z))^{-1} d\mu^{\bD} (\lambda)$.
\end{rmk}

%

%

While beyond the scope of this paper, it is of 
interest to study the behavior of the eigenvectors 
of large random matrices of heavy tailed entries
(such as $\bA_N^\sigma$ or $\bW_{N,M}$), and in 
particular, to find out if they concentrate 
on indices associated with the entries of 
extreme values or are rather ``spread-out''.

After devoting the next section 
to the
truncation and approximation tools
used in our work, we proceed to prove our 
main results, starting with the proof of 
Theorem \ref{theo-limitpoint-uniq-amir} 
in Section \ref{sec:ind}. This is 
followed by the proof of 
Theorem \ref{weakening} in Section \ref{sec:gen},
the specialization to covariance matrices
(i.e. proof of Theorem \ref{wishart-amir}) 
in Section \ref{sec:wis}
and the generalization to diagonal perturbations
(i.e. proof of Theorem \ref{weakeningD}) in
Section \ref{sec:diag}.

\section{Truncation, tightness and approximations}\label{cutoff}

As the second moment of entries of our random matrices is infinite, 
we start by providing 
appropriate truncated matrices, whose 
spectral measures approximate well (in the limit $N \to \infty$)
the spectral measures $\mun_{\bA_N}$. Specifically, let
$\bY_N^B$ denote the $N \ts N$ symmetric matrix with
entries $\sigm x_{ij} 
{\bf 1}_{|x_{ij}| < Ba_N}$ for $B>0$.
We further consider the $N \ts N$ symmetric matrix $\bY_N^\kappa$  
with entries $\sigm x_{ij} 
{\bf 1}_{|x_{ij}| < N^{\kappa}a_N}$ for $\kappa>0$,
and the corresponding normalized matrices,
$$
\bA^B_N:=a_N^{-1} \bY_N^B, \quad \bA^\kappa_N:=a_N^{-1}\bY_N^\kappa.$$
It is easy to adapt the proof 
of \cite[Lemma 2.4]{BAG6} to our setting and deduce that  
for every $\e>0$, there exists  $B(\e)$ finite and $\d(\e,B)>0$ when
$B>B(\e)$,
such that
$$\P(\mbox{rank}(\bY_N-\bY_N^B)\ge \e N)\le e^{-\d(\e,B) N} \,.$$
Likewise,
for $\kappa>0$, and $a\in]1-\alpha \kappa,1[$ there exists a finite
constant $C=C(\alpha,\kappa,a)$ such that 
$$ 
\P(\mbox{rank}(\bY_N-\bY_N^\kappa)\ge N^a)\le e^{-C N^a\log N} 
$$
(and both bounds are independent of $\sigma(\cdot,\cdot)$).
By Lidskii's theorem it then readily follows that 
\begin{equation}\label{trunc1}
\P\left(d_1(\mun_{\bA_N},\mun_{\bA_N^B})\geq 2\e\right)
\le e^{-\d(\e,B) N}\,,
\end{equation}
\begin{equation}\label{trunc2}
\P\left(d_1(\mun_{\bA_N},\mun_{\bA_N^\kappa}) \geq 2N^{a-1}\right) \le
e^{-CN^a\log N }\,,
\end{equation}
where the metric 
$$
d_1(\mu,\nu):=\sup_{ \|f\|_{\rBL} \le 1, f\uparrow}
\Big| \int f d\nu- \int f d\mu\Big|
$$
on the set $\Pa(\R)$ of Borel probability measures on $\R$
is compatible with the topology of weak convergence
(for example, see \cite[Lemma 2.1]{BAG6}), 
and throughout  $\|f\|_\rBL$ denotes the standard Bounded Lipschitz norm
on $\R$. 

Just as in \cite[Lemmas 3.1]{BAG6}, we have the following tightness result.
\begin{lem}
\label{tight}
The sequence $(\E[\mun_{\bA_N}];N\in\N)$ is tight for the topology of 
weak convergence on $\Pa(\R)$. Further, 
for every $B>0$ and $\kappa>0$, the sequences
$(\E[\mun_{\bA^B_N}]; N\in\N)$ and
$(\E[\mun_{\bA_N^\kappa}]; N\in\N)$ are also 
tight in this topology. 
\end{lem}
\begin{proof} Recall that 
\begin{equation}\label{contcov2}
\E[\frac{1}{N}\tr( (\bA_N^B)^2)]=\frac{1}{Na_N^2} \sum_{i,j=1}^N
\sigm^2 
\E[|x_{ij}|^2 
{\bf 1}_{|x_{ij}|<Ba_N}]
\end{equation}
As the latter expectation does not depend on $i,j$ 
and using the key estimate 
\begin{equation}
\label{truncatedmoments2} 
\E[|x_{ij}|^{\zeta}1_{|x_{ij}|<Ba_N}] \sim
\frac{\alpha}{\zeta-\alpha}B^{\zeta-\alpha} a_N^{\zeta} N^{-1} \,,
\end{equation}
for any $\zeta>\alpha$, we deduce that since $\sigma$ is in 
$L^2_\star([0,1]^2)$,
\begin{equation}\label{contcov3}
\lim_{N\ra\infty} \E[\frac{1}{N}\tr( (\bA_N^B)^2)]
\le \frac{\alpha}{2-\alpha} B^{2-\alpha} \|\sigma \|^2_\star < \infty \;\;.
\end{equation}
This implies the tightness of $(\E[\mun_{\bA_N^B}], N\in\N)$ which 
upon using \eqref{trunc1} and \eqref{trunc2} 
provides also the tightness of $(\E[\mun_{\bA_N}], N\in\N)$ and 
$(\E[\mun_{\bA_N^\kappa}], N\in\N)$, respectively (for more  
details, see the proof of \cite[Lemma 3.1]{BAG6}). 
\end{proof}

We next show that it suffices to prove the
convergence of the spectral measures
$\E[\mun_{\bA_N^\sigma}]$ for $\sigma(\cdot,\cdot)$ 
in any given dense subset of $L^2_\star([0,1]^2)$.
\begin{prop}\label{prop-dense}
Suppose that a sequence $(\sigma_p, p \in \N)$
converges in $L^2_\star([0,1]^2)$ towards $\sigma$ and that for all $p\in\N$
\begin{equation}\label{hyp1}
\lim_{N\ra\infty}
\E[\mun_{\bA_N^{\sigma_p}}]=\mu^{\sigma_p} \,.
\end{equation}
Then,
$\mu^{\sigma_p}$ converges weakly as $p \to \infty$ 
towards some Borel probability measure $\mu^\sigma$
and $\E[\mun_{\bA_N^{\sigma}}]$ converges weakly 
towards $\mu^\sigma$ as $N \to \infty$.
\end{prop}
\begin{proof}
Note that for some finite constant 
$c=c(\alpha,B)$, independent of $N$ and $\sigma$, 
\begin{equation}
\E[d_1(\mun_{\bA_N^{B,\sigma}}, \mun_{\bA_N^{B,\sigma_p}})]^2
\le \E\Big[\frac{1}{N}\tr\Big(
( \bA_N^{{B,\sigma}}-\bA_N^{B,\sigma_p})^2\Big)\Big]
\le c^2 \|\sigma-\sigma_p\|_\star^2  \,.
\label{approx0}
\end{equation}
Indeed, the leftmost inequality is based on 
Lidskii's theorem (see \cite[(2.16)]{BDJ}), whereas 
the rightmost one is obtained by 
an application of \eqref{contcov2}--\eqref{contcov3}
with $\sigma$ replaced by $\sigma-\sigma_p$. Next, from 
the triangle inequality for the $d_1$-metric, we have that
\begin{eqnarray*}
d_1(\E[\mun_{\bA_N^{\sigma}}],\mu^{\sigma_p})
&\le& d_1(\E[\mun_{\bA_N^{\sigma}}],\E[\mun_{\bA_N^{B,\sigma}}])
+d_1(\E[\mun_{\bA_N^{B,\sigma}}],\E[\mun_{\bA_N^{B,\sigma_p}}])\\
&&+d_1(\E[\mun_{\bA_N^{B,\sigma_p}}],\E[\mun_{\bA_N^{\sigma_p}}])
+d_1(\E[\mun_{\bA_N^{\sigma_p}}], \mu^{\sigma_p})\,.
\end{eqnarray*}
By our hypothesis \eqref{hyp1}, the last term converges to zero
as $N \to \infty$. Further, 
by \eqref{trunc1} and the boundedness and convexity
of $d_1$, we find that for some $\e(B) \to 0$ as 
$B \to \infty$, independently of $\sigma$ and $\sigma_p$,
$$
\limsup_{N\ra\infty} 
d_1(\E[\mun_{\bA_N^{\sigma}}],\E[\mun_{\bA_N^{B,\sigma}}])
+ \limsup_{N \to \infty}
d_1(\E[\mun_{\bA_N^{B,\sigma_p}}],\E[\mun_{\bA_N^{\sigma_p}}])
\le 8 \e(B)\,.
$$
Moreover, by the convexity of $d_1$ and \eqref{approx0}, we have that 
$$
\limsup_{N\ra\infty} d_1(\E[\mun_{\bA_N^{B,\sigma}}],\E[\mun_{\bA_N^{B,\sigma_p}}])
\le c(\alpha,B) \|\sigma-\sigma_p\|_\star \,.
$$
Upon combining these estimates 
we deduce that for any $p \in \N$ and $B>0$, 
\begin{equation}\label{approx3}
\limsup_{N\ra\infty}
d_1(\E[\mun_{\bA_N^{\sigma}}],\mu^{\sigma_p})\le 8 \e(B)+
c(\alpha,B)\|\sigma-\sigma_p\|_\star\,.
\end{equation}
In particular, we get the bound
$$\sup_{p,q\ge r} d_1(\mu^{\sigma_p}, \mu^{\sigma_{q}})
\le 16\e(B) +2c(\alpha,B)\d(r)^2 \,,
$$
where by hypothesis $\d(r)^2:=\sup_{p\ge r} \|\sigma-\sigma_p\|_\star$
converges to zero as $r \to \infty$.
Taking $r$ and $B$ going to infinity such that $c(\alpha,B)\le \d(r)^{-1}$
we conclude that $(\mu^{\sigma_p},p\in\N)$
is $d_1$-Cauchy and hence converges to some $\mu^\sigma \in \Pa(\R)$
(recall that $\e(B)$ and $c(\alpha,B)$ are independent 
of $\sigma$). 
By this convergence, combining \eqref{approx3}
and the triangle inequality for the $d_1$-metric, 
we deduce upon taking $p \to \infty$ and then $B \to \infty$, 
that $\E[\mun_{\bA_N^{\sigma}}]$ also
converges towards $\mu^\sigma$ as $N \to \infty$.
\end{proof}

\begin{rmk}\label{remarkcenter}
By our assumptions, when dealing with $\sigma \in \FF_\alpha$
we may and shall take in Proposition \ref{prop-dense}
some $\sigma_p\in \PC$. Since 
the rank of the matrix  $\E[ \bA_N^{\kappa, \sigma_p}]$
is then uniformly bounded in $N$, 
as in \cite[Remark 2.5]{BAG6} we  may  and shall recenter
$\bA_N^{\kappa, \sigma_p}$ 
without changing its limiting spectral distribution.
\end{rmk}

We conclude by showing an interpolation property
of $\mun_{\bA^\s_N}$ in case $\s \in \PC$. That is, 
the weak convergence of $\mun_{\bA^\s_N}$ follows
once we have it along a suitable sub-sequence $\phi(n)$.
\begin{lem}\label{lem:interp} 
Suppose $\s \in \PC$ and the increasing function
$\phi:\N \to \N$ is such that $\phi(n-1)/\phi(n) \to 1$. 
If $\mun_{\bA^\s_{\phi(n)}}$ converges weakly to some 
probability measure $\mu^\s$ then so does $\mun_{\bA^\s_N}$.
\end{lem}
\begin{proof} For any $N \in (\phi(n-1),\phi(n)]$ set
$M=\phi(n)$ and let $\wh{\bA}^\s_N$ denote the $M \times M$
dimensional matrix whose upper left $N \times N$ corner equals
$(a_N/a_M) \bA^\s_N$ and 
having zero entries everywhere else. 
Letting $0=b_0<b_1<\cdots<b_{q}=1$ denote the partition 
that corresponds to $\s \in \PC$, observe that 
$\wh{A}^\s_N(i,j)=A^\s_M(i,j)$ unless either 
$i \in (b_r N,b_r M]$ or $j \in (b_r N,b_r M]$ for some 
$r=0,1,\ldots,q$. As the latter applies for at most $(q+1)(M-N+1)$ 
values of $1 \le i \le M$ and at most $(q+1)(M-N+1)$ values
of $1 \le j \le M$, it follows that
$$
\mbox{rank}(\wh{\bA}^\s_N - \bA^\s_M) \le 2(q+1)(M-N+1) \,,
$$ 
so by Lidskii's theorem 
$$
d_1(\mun_{\wh{\bA}^\s_N},\mun_{\bA^\s_M}) \le 4(q+1)(1-\frac{N-1}{M})
\le 4(q+1)(1-\frac{\phi(n-1)}{\phi(n)}) \,,
$$
which converges to zero as $N \to \infty$ (hence $n \to \infty$).
Therefore, by the triangle inequality for the $d_1$-metric,
our assumption that  
$d_1(\mun_{\bA^\s_{\phi(n)}},\mu^\s) \to 0$ implies that 
$d_1(\mun_{\wh{\bA}^\s_N},\mu^\s) \to 0$ as $N \to \infty$. 
Next note that the eigenvalues of $\wh{\bA}^\s_N$ are 
those of $(a_N/a_M) \bA^\s_N$ augmented by $M-N$ 
zero eigenvalues. Fixing a monotone non-decreasing
bounded Lipschitz function $f(\cdot)$, 
we have thus seen that 
\begin{equation}\label{eq:inter-rel}
\int f d\mun_{\wh{\bA}^\s_N} = (1-\frac{N}{M}) f(0) +
\frac{N}{M} \int f(\beta_N x) d \mun_{\bA^\s_N} (x) \to 
\int f d\mu^\s \,, 
\end{equation}
when $N \to \infty$, 
where $1 \ge \beta_N := a_N/a_M \ge a_{\phi(n-1)}/a_{\phi(n)}$
(as both $\phi(\cdot)$ and $a_k$ are non-decreasing, 
see \eqref{normalisation}). Since $\phi(n-1)/\phi(n) \to 1$ 
the same applies for $N/M \in (\phi(n-1)/\phi(n),1]$. Further,  
$a_k=L_0(k) k^{1/\alpha}$ with 
$L_0(\cdot)$ a slowly varying function, hence also
$a_{\phi(n-1)}/a_{\phi(n)} \to 1$ when $n \to \infty$
and consequently $\beta_N \to 1$ as $N \to \infty$.
Fixing $\epsilon>0$, since $f(\cdot)$ is monotone and 
bounded, there exists $K=K(\epsilon)$
finite such that $|f(x)-f(y)| \le \epsilon$ whenever
$\min(x,y) \ge K$ or $\max(x,y) \le -K$.
Thus, for any $\beta \in (0,1]$, 
$$
\sup_{\nu \in \Pa(\R)} \int |f(x) - f(\beta x)| d \nu (x) 
\le \epsilon + \frac{K}{\beta} (1-\beta)\|f\|_{\LL} \,.
$$ 
In particular, since $\beta_N \to 1$, for any $\epsilon >0$, 
$$
\lim_{N \to \infty} | \int f d \mun_{\bA^\s_N} -
\int f(\beta_N x) d \mun_{\bA^\s_N} (x)| \le \epsilon \,,
$$
which in view of \eqref{eq:inter-rel} 
results with $\int f d \mun_{\bA^\s_N} \to \int f d\mu^\s$. 
This holds for each monotone non-decreasing
bounded Lipschitz function $f(\cdot)$, which is equivalent to 
our thesis that $\mun_{\bA^\s_N}$ converges weakly to $\mu^\s$. 
\end{proof}


\section{Induction and the limiting equations}\label{sec:ind}
We consider throughout this section $\sigma\in\PC$.
That is, there exist $0=b_0<b_1<\cdots<b_{q}=1$
and a $q \ts q$ symmetric matrix of entries 
$\sigma_{rs}$ for $1 \le r,s\le q$ such that 
\begin{equation}\label{eq:sPC}
\sigma(x,y)=\sigma_{rs}\qquad \mbox{\rm for all} \quad
(x,y) \in (b_{r-1},b_r] \ts (b_{s-1},b_s] \; .
\end{equation}
Associated with such $\sigma$ are the 
random matrix $\bA_N^\sigma$ and the 
$N\ts N$ piecewise constant matrix $\bsigma^N$ 
of entries $\sigma^N(i,j)= \sigma_{rs}$ for 
$[Nb_{r-1}] <  i \leq [Nb_{r}]$ and 
$[Nb_{s-1}] <  j \leq [Nb_{s}]$.

\subsection{Characterization of limit points}
For each $z\in\C^+=\{z\in\C: \Im(z)>0\}$ 
we define, as in \cite[Section 4]{BAG6},
the matrices
$\bG_N(z):=(z\bI_N-\bA_N)^{-1}$ and the probability 
measure $L_N^z$ on $\C$ such that for 
$f\in\Ca_b(\C)$, 
\begin{equation}\label{eq:lnz}
L_N^z(f)= \E\Big[
\frac{1}{N} \sum_{k=1}^N f\left(
G_N(z)_{kk}\right)\Big].
\end{equation}
It is useful for our purpose to represent $L_N^z$
as a weighted sum $L_N^z=\sum_{r=1}^q \Delta_{N,r} L_{N,r}^z$ where 
$L_{N,r}^z$ are the probability measures on $\C$ given by 
\begin{equation}\label{eq:lnzr}
L_{N,r}^z(f):= \E\Big[
\frac{1}{[N b_r]-[N b_{r-1}] } \sum_{k= [Nb_{r-1}]+1}^{[N b_r]} f\left(
G_N(z)_{kk}\right)\Big],
\end{equation}
and 
$\Delta_{N,r}:=N^{-1} ([N b_r]-[N b_{r-1}])\ra \D_r$
as $N \to \infty$. 
Since each term $G_N(z)_{kk}$ belongs to the compact set 
$\Db (z):=\{x\in \C^- : |x|\le |\Im(z)|^{-1}\}$,
the probability measures $L_{N,r}^z$ are supported on $\Db (z)$
for all $N \in \N$ and $1 \le r \le q$.

We denote by $\bG_N^\kappa(z)$ and $L_{N,r}^{z,\kappa}$ 
the corresponding objects 
when $\bA_N$ is replaced by the truncated matrix $\bA_N^\kappa$.
Similarly to \cite[Lemma 4.4]{BAG6} we next show that
\begin{lem}
\label{approxss} 
For $0< \kappa < \frac{1}{2(2-\alpha)}$,
any $1\le r\le q$ and Lipschitz function $f$ on $\Db (z)$,
$$\lim_{N\ra\infty}
\Big|
\E\big[L_{N,r}^{z,\kappa}(f)\big]-\E\Big[f\big( (z- \sum_{k=1}^N
\widetilde A_N^\kappa([N b_r],k)^2 G_N^\kappa(z)_{kk})^{-1}\big)\Big]
\Big|=0 \,,
$$
where $\widetilde \bA_N^\kappa$ is an independent copy
of $ \bA_N^\kappa$. 
\end{lem}
\begin{proof}
Without loss of generality, it suffices to  prove the 
lemma for $r=1$ (the general case follows by permuting indices).
To this end, let 
$\bar \bA_{N+1}^\kappa$ 
denote 
an $(N+1)\ts (N+1)$ symmetric matrix obtained by
adding to $\bA_N^\kappa$ a first row and column 
$\wt{A}_N^\kappa(0,k)=\wt{A}_N^\kappa(k,0)$
such that $(\wt{A}_N^\kappa(0,k),k\ge 1)$ is an independent 
copy of $(A_N^\kappa(1,k),k\ge 1)$ and $\wt{A}_N^\kappa(0,0)= \sigma^N(1,1) 
a_N^{-1}  x_{00} 
{\bf 1}_{|x_{00}| < N^\kappa a_N}$.
Next, consider the matrix  
$\bar \bG^\kappa_{N+1}(z) = (z\bI_{N+1}-\bar \bA_{N+1}^\kappa)^{-1}$ and let 
$\bar L_{N+1,1}^{z,\kappa}$ denote the empirical measure of 
$\{ \bar G^\kappa_{N+1}(z)_{kk}$,
$0\le k \leq  [N b_1]\}$. The invariance of the law
of $\bar \bA_{N+1}^\kappa$ with respect to
symmetric permutations of its first $[N b_1]+1$ rows 
and columns implies that 
$\{\bar G^\kappa_{N+1}(z)_{kk}$, $0\le k \leq  [N b_1]\}$ 
are identically distributed, hence for any $f\in \Ca_b(\Db (z))$, 
\begin{equation}\label{eq:lbrn00}
\E[ \bar L_{N+1,1}^{z,\kappa}(f)]= \E[f(\bar G^\kappa_{N+1}(z)_{00})]\,.
\end{equation}
As in \cite{BAG6}, the key to our proof is Schur's complement formula
$$
\bar G^\kappa_{N+1} (z)_{00}
=\big(z-\wt{A}_N^\kappa({0,0}) -\sum_{k,l=1}^N 
\wt{A}_N^\kappa(0,k) \wt{A}_N^\kappa(l,0)
G_N^\kappa(z)_{kl}\big)^{-1}\,,
$$
from which we thus get that 
\begin{equation}\label{equalitylemma1}
\E[ \bar L_{N+1,1}^{z,\kappa}(f)] 
=\E
\Big[f(\big(z- \wt{A}_N^\kappa({0,0})-\sum_{k,l=1}^N \wt{A}_N^\kappa(0,k) 
\wt{A}_N^\kappa(l,0)
G_N^\kappa (z)_{kl}\big)^{-1})\Big]. 
\end{equation}
Recall that the entries of $\wt{\bA}_N^\kappa$ are centered 
(see Remark \ref{remarkcenter}),
and independent of the matrix $\bG_N^\kappa(z)$.
Further, as the entries of 
the matrix $\bsigma^N$ are uniformly bounded, the statement and
proof of \cite[Lemma 4.3]{BAG6} extends readily to our setting,
showing that the off diagonal terms in the right hand
side of \eqref{equalitylemma1} are small with overwhelming
probability (this is simply based on a computation
of the variance of this term, which is possible thanks to 
the cut-off $\kappa$). As shown in 
the proof of \cite[Lemma 4.4]{BAG6},
this allows us
to neglect the terms $\wt{A}_N^\kappa(0,0)$ and  
$\sum_{k\neq l} \wt{A}_N^\kappa(0,k) \wt{A}_N^\kappa(l,0) G_N^\kappa(z)_{kl}$ 
in \eqref{equalitylemma1}, resulting with 
\begin{equation}
\label{approx2}
\lim_{N\ra\infty}
\Big|\E[\bar L_{N+1,1}^{z,\kappa}(f)]- \E\Big[f\big( (z- \sum_{k=1}^N
\wt{A}_N^\kappa(0,k)^2 G_N^\kappa(z)_{kk})^{-1}\big)\Big]\Big|=0 \,.
\end{equation}
Further, with 
$\bsigma^N$ uniformly bounded, adapting the proof of 
\cite[Lemma 4.1]{BAG6} to our setting, we deduce that 
$$
\lim_{N\ra\infty}\P
( d_1(L_{N,1}^{z,\kappa},\bar L_{N+1,1}^{z,\kappa})>N^{-\eta})=0\,,
$$
for any $0<\eta< \frac{1}{2}(1-\kappa(2-\alpha))$. 
Consequently, $|\E[\bar L_{N+1,1}^{z,\kappa}(f)]-
\E[ L_{N,1}^{z,	\kappa}(f)]| \to 0$ as $N \to \infty$ 
and \eqref{approx2} finishes the  proof of the lemma.
\end{proof}

Identifying $\C$ with $\R^2$, recall 
\cite[Definition 5.1]{BAG6}. Namely,
\begin{defi}
Given $\alpha \in (0,2)$ and a compactly supported 
probability measure $\mu $ on $\C$, let $P^\mu$
denote the probability measure on $\C$ whose 
characteristic function at $\bt \in \R^2$ is  
$$
\int_{\R^2} e^{i\langle \bt,\bx\rangle} dP^{\mu}(\bx)
= \exp [-v_{\mu,\frac{\alpha}{2}}(\bt)^{\frac{\alpha}{2}}
(1-i\beta_{\mu,\frac{\alpha}{2}}(\bt)
\tan(\frac{\pi \alpha}{4}))] \,,
$$
where
\begin{eqnarray*}
v_{\mu,\alpha}(\bt)&=&[v_{\alpha}^{-1}
\int|\langle \bt,\bz\rangle|^{\alpha}d\mu(\bz)]^{1/\alpha} \,,\\
v_{\alpha}^{-1}&=& \int_0^{\infty} \frac{\sin x}{x^{\alpha}}dx=
\frac{\Gamma(2-\alpha)\cos(\frac{\pi\alpha}{2})}{1-\alpha} \,,\\
\beta_{\mu,\alpha}(\bt)&=&\frac{\int|\langle \bt, \bz \rangle|^{\alpha}
{\rm sign}(\langle \bt, \bz \rangle) d\mu(\bz)}{
\int|\langle \bt,\bz\rangle|^{\alpha}d\mu(\bz)} \,,
\end{eqnarray*}
and $\beta_{\mu,\alpha}(\bt)=0$ whenever $v_{\mu,\alpha}(\bt)=0$.
In particular, if $\mu$ is supported in 
the closure of $\C^-$, then so does $P^\mu$. 
\end{defi}

Equipped with this definition, our next proposition 
characterizes the set of possible limit points 
of $\{\E[L_{N,r}^{z,\kappa}], 1\le r\le q\}$. 
\begin{prop}\label{prop-limitpoint}
For $0< \kappa < \frac{1}{2(2-\alpha)}$ and $z\in\C^+$,
any limit point $(\mu^z_r, 1\le r\le q)$ 
of the sequence $\{
(\E[L_{N,r}^{z,\kappa}], 1\le r\le q)$,  $N\in\N \}$ 
consists of probability measures on $\Db (z)$ that 
satisfy the system of equations
\begin{equation}\label{limitpointeq}
\int f d\mu^z_r =\int f\Big( (z- \sum_{s=1}^q
\sigma_{rs}^2  \Delta_s^{\frac{2}{\alpha}} x_s)^{-1}\Big) 
\prod_{s=1}^qdP^{\mu^z_s}(x_s)
\end{equation}
for $r\in\{1,\ldots,q\}$ 
and every bounded continuous function $f$ on $\Db (z)$.
\end{prop}

The following concentration result is 
key to the proof of Proposition \ref{prop-limitpoint}.
\begin{lem}\label{concentration}
For $\kappa\in (0,\frac{1}{2-\alpha})$ let $\e= 1 -
\kappa(2-\alpha) >0$. There exists $c<\infty$ so that for 
$z\in \C^+$, $s\in \{1,\ldots, q\}$, $\de>0$,  $N \in \N$ and 
any Lipschitz function $f$ on $\Db (z)$, 
$$
\P\left(\left| L_{N,s}^{z,\kappa}(f)-
\E[L_{N,s}^{z,\kappa}(f)]\right|\ge \d\right)
\le \frac{c\|f\|_\rBL^2}{|\Im(z)|^4\d^2}  N^{-\e} \,,
$$
with $\|f\|_\rBL$ denoting here the Bounded Lipschitz 
norm of $f$ restricted to $\Db (z)$.  
\end{lem}
\begin{proof} Fixing $s\in  \{1,\ldots, q\}$ and $z \in \C^+$, note that
the value of $f$ outside the compact set $\Db (z)$ on which 
all probability measures $L_{N,s}^{z,\kappa}$ are supported, 
is irrelevant. We thus assume without loss of generality that 
$f$ is bounded and continuously differentiable and as in the
proof of \cite[Lemma 5.4]{BAG6}, let
$$
F_N(\bA):= 
L_{N,s}^{z,\kappa}(f)=\frac{1}{N}\sum_{k=[b_{s-1}N]+1}^{[b_{s}N]} 
f(G^\kappa_N(z)_{kk}) \,,
$$
a smooth function of the $n=N(N-1)/2$ independent, centered,
random variables $A^\kappa_N({k,l})$ for $1 \le k \le l \le N$.
By a classical martingale decomposition we see that 
\begin{equation}
\label{danslebureau}
\E[ (F_N-\E[F_N])^2] \le\sum_{1\le i\le j\le N}
 \|\partial_{A(i,j)}F_N\|_\infty^2 
\E[ (A^\kappa_N({i,j})-\E[A^\kappa_N({i,j})])^2] \,.
\end{equation}
Moreover, similarly to the proof of \cite[Lemma 5.4]{BAG6} we
have here that 
\begin{eqnarray*}
\partial_{A(m,l)} F_N (\bA)&\!\!\!=\!\!\!&
\frac{1}{N}\sum_{k=[N b_{s-1}]+1}^{[N b_s]}
\!\!\! f'(G_N^\kappa (z)_{kk})(G_N^\kappa(z)_{kl}G_N^\kappa(z)_{mk}
+G_N^\kappa (z)_{km}G_N^\kappa (z)_{lk})\\
&\!\!\!=\!\!\!&\frac{1}{N}\left(
[\bG_N^\kappa(z) \bD_s(f')\bG_N^\kappa(z)]_{ml}+[\bG_N^\kappa(z) \bD_s(f')
\bG_N^\kappa(z)]_{lm}\right)
\end{eqnarray*}
with $\bD_s(f')$ the $N$-dimensional diagonal matrix of entries
$$
D_s(f')_{kk} := f'(G_N^\kappa(z)_{kk}) 
{\bf 1}_{[N b_{s-1}] < k\le [N b_{s}]}
\,.
$$
As the spectral radius of $\bG_N^\kappa (z) \bD_s(f') \bG_N^\kappa (z)$ is 
bounded by $\|f'\|_\infty /  |\Im(z) |^2$, the same applies
for each entry of 
this matrix. By the preceding, such bounds imply that 
$$
\sup_{i,j} \|\partial_{A(i,j)}F_N\|_\infty 
\leq 2\|f\|_\rBL (N |\Im(z) |^2)^{-1} \,.
$$
Further, with $\bsigma^N$ uniformly bounded, 
from \eqref{truncatedmoments2} (for $\zeta=2$),
we get that for some $c_0$ finite and all $N$, 
$$
\sup_{1 \leq i \leq j \leq N} 
\E[|A^{\kappa}_N(i,j)|^2] \le c_0 N^{\kappa(2-\alpha)-1}\,.
$$
As $\e=1-\kappa(2-\alpha)>0$,
substituting these bounds into \eqref{danslebureau} we find  that 
$$
\E[ (F_N-\E[F_N])^2] \le 4 c_0 \|f\|_\rBL^2 |\Im(z)|^{-4} 
N^{-\epsilon} \,,
$$ 
and conclude the proof by Chebychev's inequality.
\end{proof}

\begin{proof}[Proof of Proposition \ref{prop-limitpoint}]
The sequence  of $q$-tuples of probability
measures $(\E[L_{N,r}^{z,\kappa}]$, $1\le r\le q)_{N\in\N}$, 
each supported in the compact set $\Db (z)$, 
is clearly tight. Considering a subsequence
$(\E[L_{\phi(N),r}^{z,\kappa}], 1\le r\le q)_{N\in\N}$ 
that converges weakly to a limit point $(\mu^z_r,
1\le r\le q)$, passing to a further 
subsequence still denoted $\phi(N)$
we have by Lemma \ref{concentration}
that $(L_{\phi(N),r}^{z,\kappa}, 1\le r\le q)_{N\in\N}$ 
also converges almost surely to $(\mu^z_r,
1\le r\le q)$, a $q$-tuple of probability measures on
$\Db (z)$. 

By Lemma \ref{approxss}, fixing $r\in\{1,\ldots,q\}$,
it suffices to show that 
$$U_N(z,r):=\sum_{k=1}^N
\widetilde A_N^\kappa([b_rN],k)^2 G_N^\kappa(z)_{kk}$$ 
is such that $U_{\phi(N)}(z,r)$ converges 
in law towards $\sum_{s=1}^q
\sigma_{rs}^2  \Delta_s^{2/\alpha} x_s$ where $(x_s, 1\le s\le q)$
are independent, with $x_s \in \C$ distributed according to 
$P^{\mu^z_s}$ for $s=1,\ldots,q$.

Note that $U_N(z,r) = \sum_{s=1}^q\sigma_{rs}^2 W_N(z,s)$, where
$$
W_N(z,s):=\sum_{k=[N b_{s-1}]+1}^{[N b_{s}]}
\widehat A_N^\kappa([b_rN],k)^2 G_N^\kappa(z)_{kk}\,,
$$
and the i.i.d. random variables 
$\widehat A_N^\kappa([b_rN],k)=\widetilde A_N^\kappa([b_rN],k)/ \sigma_{rs}$
are independent of $\bG_N^\kappa(z)$ and correspond to taking $\sigma \equiv 1$.
Next let
$$
a_N(s)=\inf\{ u: \P(|x_{ij}|\ge u)\le \frac{1}{N \Delta_{N,s}}\} \,,
$$
noting that by \eqref{stabledomain},
\begin{equation}\label{convaN}
\lim_{N\ra\infty}\frac{a_N(s)}{a_N} =\Delta_s^{1/\alpha}\,.
\end{equation}
Further, applying 
\cite[Theorem 10.4]{BAG6}   
for $X_k=\wt{x}_{[Nb_r] k}^2$, $\wt a_N=a_N(s)^2$ and 
$\ell(N)=(a_N/a_N(s))^2 N^{2\kappa}\ra \infty$, 
on the subsequence $\phi(N)$ and subject to the event that $
L_{\phi(N),s}^{z,\kappa}$ 
converges to $\mu^z_s$, we deduce that
$(a_N/a_N(s))^2 W_N(z,s)$ 
converges in law to $P^{\mu^z_s}$.
By the conditional 
independence of $W_N(z,s)$ for $1\le s\le q$ (per fixed 
$G_N^\kappa(z)$),
and \eqref{convaN} we arrive at the stated convergence in law
of $U_{\phi(N)}(z,r)$. 
\end{proof}

We next derive the analog of \cite[Theorem 5.5]{BAG6}. 
\begin{prop}\label{projectionoflimitpoint}
For $0< \kappa < \frac{1}{2(2-\alpha)}$ 
any subsequence of the functions  
$(X_{N,r} (z) := \E[L_{N,r}^{z,\kappa}(x^{\alpha/2})],1\le r\le q)$ 
from $\C^+$ to $\C^q$ 
has at least one limit point $(X_{r}(z),1 \le r \le q)$ 
such that 
$z \mapsto X_{r}(z)$ are analytic in $\C^+$, 
$|X_r(z)|\le (\Im(z))^{-\alpha/2}$ and for all $z \in \C^+$,
\begin{equation}\label{cocottes}
X_r(z)= C(\alpha)\int_0^\infty t^{-1} (it)^{\frac{\alpha}{2}}
e^{itz} \exp\{- (it)^{\frac{\alpha}{2}} \wh{X}_r(z) \} \, dt \,,
\end{equation}
with $C(\alpha)=\frac{e^{-i\frac{\pi\alpha}{2}}}{\Gamma(\frac{\alpha}{2})}$
and 
\begin{equation}\label{eq:yrdef}
\wh{X}_r(z) := \Gamma(1-\frac{\alpha}{2}) 
\sum_{s=1}^q |\sigma_{rs}|^\alpha \D_s X_s(z) \,.
\end{equation}
\end{prop}
\begin{proof}
The proof is an easy adaptation of \cite[Theorem 5.5]{BAG6}.
In fact, for each $1\le r\le q$, the analytic functions 
$X_{N,r} (z)$ on $\C^+$ are 
uniformly bounded by $(\Im(z))^{-\alpha/2}$ (hence uniformly 
bounded on compacts). 
Consequently, by Montel's theorem, 
any subsequence $(X_{\phi(N),r}(z), 1 \le r \le q)$ 
has a limit point
$(X_r(z), 1 \le r \le q)$ 
(with respect to uniform convergence on compacts), 
consisting of analytic functions on $\C^+$ 
(c.f. \cite[Theorem 17.21]{Chae}),
that obviously are also bounded by $(\Im(z))^{-\alpha/2}$. 
Fixing $z \in \C^+$ and passing to a further sub-subsequence 
along which the compactly supported probability measures
$\E[L_{N,r}^{z,\kappa}]$ converge weakly to 
$\mu^z_r$ for all $1 \leq r \leq q$, it follows by 
definition that 
$X_r(z)= \int x^{\frac{\alpha}{2}}  d\mu^z_r(x)$
(as $x \mapsto x^{\alpha/2}$ is in $\Ca_b(\Db (z))$).
Next, we prove 
\eqref{cocottes} by applying  \cite[Lemma 5.6]{BAG6} which states that
for all $z\in\C^+$,
\begin{equation}\label{thecalculus}
z^{-\frac{\alpha}{2}}
= C(\alpha)\int_0^\infty t^{-1} (it)^{\frac{\alpha}{2}} e^{it z} dt \,.
\end{equation}
Indeed, combining \eqref{limitpointeq} and \eqref{thecalculus}
we see that 
\begin{eqnarray*}
X_r(z)&=& \int
\big(z- \sum_{s=1}^q
\sigma_{rs}^2 \Delta_s^{\frac{2}{\alpha}} x_s\big)^{-\frac{\alpha}{2}} 
 \, \prod_{s=1}^qdP^{\mu^z_s}(x_s)\\
&=&  C(\alpha)\int\int_0^\infty 
t^{-1} (it)^{\frac{\alpha}{2}} \exp\{it (z- \sum_{s=1}^q
\sigma_{rs}^2  \Delta_s^{\frac{2}{\alpha}} x_s)\} dt 
\, \prod_{s=1}^qdP^{\mu^z_s}(x_s) \,.
\end{eqnarray*}
Recall \cite[Theorem 10.5]{BAG6} that for
$\alpha \in (0,2)$ and
any probability measure $\nu$ compactly supported in the closure
of $\C^-$, 
\begin{equation}\label{cocott}
\int e^{-itx} dP^{\nu}(x)= \exp (-\Gamma(1-\frac{\alpha}{2})(it)^{\frac{\alpha}{2}}\int
x^{\frac{\alpha}{2}}d\nu(x))\,.
\end{equation}
Since $z\in\C^+$ and $\Im (x_s) \leq 0$, by  
Fubini's theorem and \eqref{cocott} we deduce that
\begin{eqnarray*}
X_r(z)&=& C(\alpha)\int_0^\infty 
t^{-1} (it)^{\frac{\alpha}{2}} e^{itz}\prod_{s=1}^q \Big(\int  
\exp\{ -it \sigma_{rs}^2  \Delta_s^{\frac{2}{\alpha}} x_s \}
dP^{\mu^z_s}(x_s) \, \Big) dt\\
&=& C(\alpha)\int_0^\infty 
t^{-1} (it)^{\frac{\alpha}{2}} e^{itz}\prod_{s=1}^q
\exp\{- \Gamma(1-\frac{\alpha}{2}) (it)^{\frac{\alpha}{2}}
|\sigma_{rs}|^\alpha \D_s X_s(z)\}
dt \,,
\end{eqnarray*}
as claimed. 
\end{proof}

\subsection{Properties of the functions $(X_r,1\le r\le q)$}

We provide now key information about 
$X_r(z)$ of Proposition \ref{projectionoflimitpoint}.
\begin{lem}\label{boundg}
For $0< \kappa < \frac{1}{2(2-\alpha)}$, $z \in \C^+$, if
$X_s(z)$ is as in Proposition \ref{projectionoflimitpoint}
and $a_s$ are non-negative for 
$s \in \{1,\ldots,q\}$, then 
$(-z)^{-\frac{\alpha}{2}} \sum_{s=1}^q a_s X_s(z)$ is 
in the set  
${\mathcal K}_\alpha:=
\{Re^{i\theta}:
|\theta|\le \frac{\alpha\pi}{2}, \, R \geq 0\}$ on which  
for each $\beta>0$, the entire function 
\begin{equation}\label{defgab}
g_{\alpha,\beta}(y): =
\int_0^\infty t^{\frac{\beta}{2}-1} e^{-t} \exp\{-
t^{\frac{\alpha}{2}} y\} dt \,, 
\end{equation}
is uniformly bounded. In particular, this applies
to $g_\alpha=g_{\alpha,\alpha}$, to $h_\alpha=g_{\alpha,2}$
and their derivatives of all order.  
\end{lem}
\begin{proof} Recall that for $z \in \C^+$ 
the measures $L^{z,\kappa}_{N,s}$ are each supported on $\C^{-}$. 
Hence, by definition 
each of the functions $X_{N,s}(z)$
is in the closed cone 
\begin{equation}\label{whKa}
\wh{\KK}_\alpha := \{R_0 e^{i \wh{\theta}}:
-\frac{\alpha\pi}{2} \le \wh{\theta} \le 0 ,\, R_0 \geq 0\} \, ,
\end{equation}
and thus so is any limit point $X_s(z)$ of $X_{N,s}(z)$.
Setting $w := (-z)^{-\frac{\alpha}{2}} 
\sum_{s=1}^q a_s X_s(z)$, it thus follows that 
for any $z \in \C^{+}$ and non-negative $a_s$,
\begin{equation}\label{eq:amir-arg-id}
0 \le \arg(w) + \frac{\alpha}{2} \arg(z) \le \frac{\alpha\pi}{2} \,. 
\end{equation}
In particular, $w \in {\mathcal K}_\alpha$, as claimed. 
Key to the boundedness of $g_{\alpha,\beta}(\cdot)$ on this set is the
identity of \cite[equation (40)]{BAG6}, where it is shown that 
\begin{equation}\label{eq:gziden}
(-z)^{-\beta/2} g_{\alpha,\beta}(y)
= \int_0^\infty t^{-1} (it)^{\frac{\beta}{2}} e^{it
z}\exp[-(-z)^{\frac{\alpha}{2}}(it)^{\frac{\alpha}{2}}y] dt\,,
\end{equation}
for any $z\in\C^+$ and $y \in \C$. Indeed, for each $\a \in (0,2)$ 
set $\eta=\eta(\a) \in (0,\pi/2]$ small enough so
$$
\varphi := \frac{\pi \a}{4} + \frac{\a}{2} \eta < \frac{\pi}{2}
$$
and let $z=e^{i \eta} \in \C^+$ when $\Im(y) \ge 0$ 
while $z=e^{i(\pi-\eta)} \in \C^+$ otherwise. 
Either way, $\Im(z) = \sin(\eta)>0$ and if   
$y = R e^{i \theta} \in \KK_\a$, that is $|\theta|\le \alpha\pi/2$, then 
$$
\Re \Big( (-z)^{\frac{\alpha}{2}}(i)^{\frac{\alpha}{2}} y \, \Big)
= R \cos(|\theta| -\frac{\pi\a}{4} + \frac{\a}{2} \eta) \ge 
R \cos(\varphi) >0\,.
$$
Setting $\xi:=\xi(\alpha) = \cos(\varphi)/(\sin(\eta))^{\a/2}>0$
we thus deduce from \eqref{eq:gziden} that for any $\beta>0$,
\begin{eqnarray}\label{eq:serban-new-bd}
|g_{\alpha,\beta}(y)| &\leq& 
\int_0^\infty t^{\frac{\beta}{2}-1} e^{-t \sin(\eta)} 
\exp[-t^{\frac{\alpha}{2}}|y| \cos(\varphi)] dt \nonumber \\
&=& (\sin (\eta))^{-\beta/2} g_{\a,\beta} ( \xi |y|) \le 
(\sin (\eta))^{-\beta/2} g_{\a,\beta} (0) \,, 
\end{eqnarray}
is uniformly bounded on ${\mathcal K}_\alpha$. 
\end{proof}

Recall that a mapping $\uu{f}: \Ub \mapsto \C^q$
defined on some open $\Ub \subseteq \C^n$ 
is holomorphic on $\Ub$ if each of its coordinates 
admits a convergent power series expansion around each point of $\Ub$. 
Proposition \ref{projectionoflimitpoint} suggests 
viewing $(X_r(z), 1 \le r \le q)$ as an implicit mapping from
$\C^+$ into $\C^q$ 
that is defined in terms of the zero set of
the holomorphic $\uu{f}=(f_r(z,w_1,\dots,w_q),1\leq r\leq q)$, where  
$$
f_r(z,w_1,\dots,w_q)=w_r-C(\alpha)\int_0^\infty t^{-1}(it)^\frac{\alpha}{2}
e^{itz}\exp\{-(it)^\frac{\alpha}{2} \sum_{s=1}^q c_{rs} w_s\}\,dt,
$$
and $c_{rs}=\Gamma\left(1-\frac{\alpha}{2}\right)
|\sigma_{rs}|^\alpha\Delta_s$. Key properties 
of $({X}_r(z) , 1\leq r\leq q)$ are then consequences 
of the rich theory of zero sets 
of holomorphic mappings. 
We shall employ this strategy, 
but for $\uu{Y}(z) \equiv (Y_1(z),\ldots,Y_q(z))$ 
where $Y_r(z):= (-z)^{-\frac{\alpha}{2}} \wh{X}_r(z)$ and
$\wh{X}_r(z)$ is given by \eqref{eq:yrdef}. Indeed, our  
next result, extending \cite[Theorem 6.1]{BAG6},
characterizes $\uu{Y}(z)$ as implicitly defined 
for $u=z^{-\alpha}$ via $u \mapsto \uu{V}(u)$ such that
\begin{equation}\label{poiu}
\uu{F}(u,\underline{V}(u))=\uu{0}\,.
\end{equation}
With $a_{rs}=C_\alpha |\sigma_{rs}|^\alpha\D_s$,
the holomorphic mapping $\uu{F}:\C \times \C^{q} \mapsto \C^q$ 
is given for $u\in\C$ and $\underline{y}=(y_1,\ldots, y_q)\in\C^q$ by
\begin{equation}\label{eq:Fdef}
F_r(u,\underline{y})=y_r-u\sum_{s=1}^q a_{rs}
g_\alpha(y_s)\qquad 1\le r\le q
\end{equation}

\begin{prop}\label{uniqueprop}
Setting $\Eaa :=\{u \in\C : -\pi \a < \arg(u) < 0 \}$,  
there exist 
$\varepsilon=\varepsilon(\sigma)>0$  
and a unique analytic solution $\uu{y}=\uu{V}(u)$
of $\uu{F}(u,\uu{y})= \uu{0}$ on the open set 
$\Eaae := \Eaa \cup \Ball (0,\varepsilon)$. 
Further, 
there exists a unique collection of 
analytic functions $(X_r(z), 1\le r\le q)$ on $\C^+$ 
such that $|X_r(z)| \le (\Im(z))^{-\frac{\alpha}{2}}$ 
and for which \eqref{cocottes} holds. 
The functions $Y_r(z)= (-z)^{-\frac{\alpha}{2}} \wh{X}_r(z)$
are then the unique solution of \eqref{systemeqY}
analytic on $z\in \C^+$ and each tending to zero as 
$|z| \to \infty$. 
Moreover, 
$Y_r(z) = V_r(z^{-\alpha}) \in {\mathcal K}_\alpha$ 
are for $r=1,\ldots,q$ such that $Y_r(-\oo{z})=\oo{Y_r}(z)$ 
and have an analytic continuation through $(R,\infty)$ 
for some finite $R=R(\sigma)$, whereas
$z^{\frac{\alpha}{2}} X_r(z)$ (hence $z^\alpha Y_r(z)$),
are uniformly bounded on $\C^+$. 
\end{prop}

\begin{proof}
First, with $(-z)^{\frac{\alpha}{2}}(-z)^{-\frac{\alpha}{2}}=1$, 
we deduce from \eqref{eq:gziden} that \eqref{cocottes} is equivalent to 
\begin{equation}\label{eq:Xridt}
X_r(z)
=C(\alpha)(-z)^{-\frac{\alpha}{2}}g_{\alpha,\alpha} (Y_r(z)) \,,
\end{equation}
which in combination with \eqref{eq:yrdef} shows that 
$(Y_r(z), 1\le r \le q)$ satisfies \eqref{systemeqY}.
The existence of analytic solutions $(X_r(z), 1\le r\le q)$ 
and $(Y_r(z), 1\le r\le q)$ such that $|X_r(z)|\le 
(\Im (z))^{-\frac{\alpha}{2}}$  
is thus obvious from Proposition \ref{projectionoflimitpoint}.
This solution of \eqref{systemeqY}
consists by Lemma \ref{boundg} of analytic functions from $\C^+$ to 
${\mathcal K}_\alpha$. 
Further, by the boundedness of $g_\alpha(\cdot)$ on $\KK_\a$ 
we know that $|X_r(z)| \le \kappa |z|^{-\alpha/2}$
and $|Y_r(z)|\le \kappa |z|^{-\alpha}$ for some finite constant 
$\kappa$,
all $z \in \C^+$ and $r \in \{1,\ldots,q\}$. 

We turn to prove the uniqueness of the analytic solution
of \eqref{systemeqY} tending to zero as 
$\Im(z) \to \infty$ (hence the uniqueness of such solutions
tending to zero as $|z| \to \infty$). To this end,
considering $\uu{F}$ of (\ref{eq:Fdef}) note that
$\uu{F}(0,\uu{0})=\uu{0}$ and 
the complex Jacobian matrix of $\uu{y} \mapsto \uu{F}(0,\uu{y})$ 
at $\uu{y}=\uu{0}$ has a non-zero determinant
(since $\partial_{y_s} F_r(0,\uu{0})=\delta_{rs}$, with 
determinant one). Consequently,  
by the local implicit function theorem 
there are positive 
constants $\varepsilon$,  
$\delta$ and an analytic solution $\uu{y}=\uu{V}(u)$
of $\uu{F}(u,\uu{y})= \uu{0}$ on $\Ball (0,\varepsilon)$ 
which for any $|u| < \varepsilon$ is also 
the unique solution with $\|\uu{y}\| < \delta$.
Identifying $\C^+$ with $\Eaa$ via the analytic function $u=z^{-\alpha}$,
note that $\uu{Y}(z)
$ solves \eqref{systemeqY} for $z \in \C^+$ if and only if 
$\uu{V}(u)=\uu{Y}(z)$ satisfies \eqref{poiu} for 
$u \in \Eaa$. Consequently, setting $R = \varepsilon^{-1/\alpha}$ finite,
any two solutions $\uu{Y}_i(z)$, $i=1,2$ of \eqref{systemeqY} that
tend to zero as $\Im(z) \to \infty$ coincide once $\Im(z) > R$ 
is large enough to assure that 
$\max_{i=1,2} \|\underline{Y}_i(z)\| < \delta$.
The uniqueness of the analytic solution $z \mapsto \uu{Y}(z)$ 
of \eqref{systemeqY} on $\C^+$ tending to zero as $\Im(z) \to \infty$ 
then follows by the identity theorem. 
By \eqref{eq:Xridt} this implies also the uniqueness of the 
solution of \eqref{cocottes} which is analytic and bounded by 
$(\Im(z))^{-\frac{\alpha}{2}}$ throughout $\C^+$.
Moreover, by the identity theorem, $u \mapsto \uu{V}(u)$ extends
uniquely to an analytic solution of (\ref{poiu}) on 
$\Eaae$ and $\uu{Y}(z) = \uu{V}(z^{-\alpha})$ 
has an analytic extension through $(R,\infty)$.

Next, recall that $\bA_N^{\kappa,-\s}=-\bA_N^{\kappa,\s}$ 
are real-valued matrices, hence by definition  
$\bG_N^{\kappa,-\s} (z) = -\oo{\bG}_N^{\kappa,\s} (-\oo{z})$
for any $z \in \C^+$, implying by \eqref{eq:lnzr} that 
$L_{N,s}^{z,\kappa,-\s}(f(x))=L_{N,s}^{-\oo{z},\kappa,\s}(f(-\oo{x}))$.
If $x \in \Db (z)$ then so is $-\oo{x}$ and 
$\oo{x^{\alpha/2}} = i^\a (-\oo{x})^{\alpha/2}$. It thus follows 
from Proposition \ref{projectionoflimitpoint} that 
$\oo{X}_s^{-\s}(z)=i^\a X_s^\s(-\oo{z})$ for any $z \in \C^+$ and
$1 \le s \le q$. 
Since 
$(X_r^\s (z), 1\le r\le q)$ are uniquely determined 
by the equations \eqref{cocottes} which are invariant under 
$\s \mapsto -\s$ and 
$\oo{(-z)^{\alpha/2}} = i^\a (\oo{z})^{\alpha/2}$
for all $z \in \C^+$,
we thus deduce from \eqref{eq:yrdef} that 
$\oo{Y_r}(z)= Y_r(-\oo{z})$ for all
$1 \le r \le q$ and $z \in \C^+$. 
\end{proof}

To recap, for some $\varepsilon>0$ we got the existence of 
a unique analytic solution $\uu{y}=\uu{V}(u)$
of $\underline{F}(u,\underline{y})= \underline{0}$ 
on $\Eaae$ for the holomorphic mapping 
$\uu{F} : \C \times \C^q \mapsto \C^q$ of \eqref{eq:Fdef}.  
We proceed to show that $\uu{V}(u)$ 
has a continuous algebraic extension to $\oEaae$,
and in particular to $(0,\infty)$
(by algebraic extension we mean that 
\eqref{poiu} holds throughout $\oEaae$).
As we show in the sequel, this 
yields the claimed continuity of the density $\rho^\s$ in Theorem 
\ref{theo-limitpoint-uniq-amir}. 

To this end, recall 
that $\Mb\subseteq\C^n$ is an embedded
complex manifold (in short, a manifold),
of dimension $p$ if for each $\uu{a}\in \Mb$ there exist a
neighborhood $\Ub$ of $\uu{a}$ in $\mathbb C^n$ and a 
holomorphic mapping $\uu{f} : \Ub \mapsto \C^{n-p}$ such that 
$\Mb\cap \Ub = \{\uu{z} \in \Ub \colon \uu{f}(\uu{z})=\uu{0} \}$ 
and the complex Jacobian matrix of 
$\uu{f}(\cdot)$ is of rank $n-p$ at $\uu{a}$ 
(in short, rank$_{\uu{a}} (\uu{f})=n-p$, 
c.f \cite[Definition 2, Section A.2.2]{chirka}). 
Indeed, our claim is merely an application of
the following general extension result for
the mapping $\uu{F}$ of \eqref{eq:Fdef}, 
taking $u_0=0$ in the nonempty open simply connected set 
$\OO=\Eaae$ of piecewise smooth boundary.
\begin{prop}\label{prop:analytic}
Suppose $\uu{F}:\C \times \C^q \mapsto \C^q$ is a 
holomorphic mapping and $\uu{F}(u,\uu{V}(u))=\uu{0}$ for 
analytic $\uu{V} : \OO \mapsto \C^q$ and a nonempty 
open connected $\OO \subseteq \C$. 
Suppose further that the graph  
\begin{equation}\label{eq:vbdef}
\Vb :=\{(u,\underline{V}(u))\colon u\in\OO\}
\end{equation} 
of $\uu{V}$ is a one-dimensional complex manifold and 
the Jacobian determinant ${\rm det}[\partial_{\uu{y}}\uu{F}]$
is non-zero at some $\uu{v}_0=(u_0,\uu{V}(u_0))$ with 
$u_0 \in \OO$. Then, $\uu{V}(\cdot)$ has a continuous extension at  
boundary points $x \in \oo{\OO}$ where $\OO$ is locally connected
and $\uu{V}$ is locally uniformly bounded 
(i.e. $\OO \cup \{x\}$ admits 
a local basis of connected relative neighborhoods
and $\uu{V}$ is uniformly bounded on $U \cap \OO$
for some neighborhood $U$ of $x$ in $\C$). Moreover,
$\uu{F}(x,\uu{V}(x))=\uu{0}$ at any such point.
\end{prop}

Deferring the proof of Proposition \ref{prop:analytic} to the
end of this section, we next collect all properties 
needed for applying it in our setting. 
\begin{lem}\label{lem:extension}
Assuming $\s \not \equiv 0$, the mapping 
$u\mapsto\underline{V}(u)$ of Proposition 
\ref{uniqueprop} is injective on $\Eaae$
(and consequently, so is the map 
$z \mapsto \uu{Y}(z)=\uu{V}(z^{-\alpha})$).
Further, in this case 
$\Vb :=\{(u,\underline{V}(u))\colon u\in\Eaae\}$
is a one-dimensional complex manifold 
containing the point $(0,\uu{0})$ where 
$[\partial_{\uu{y}}\uu{F}]$ is the
identity matrix, and $\|\uu{V}(u)\|_2 \le K|u|$ 
for some finite constant $K =K(\sigma)$ and all $u \in \Eaae$. 
\end{lem}
\begin{proof} First note that if 
$\underline{F}(u,\underline{y})=\underline{F}(\wt{u},
\underline{y})=\uu{0}$ for some $\uu{y} \ne \uu{0}$
then by \eqref{eq:Fdef} necessarily $u=\wt{u}$.
Further, by excluding $\s \equiv 0$ we made sure
that if $ \underline{F}(u,\underline{0})=\underline{0}$
then $u=0$ (since $g_\a(0)>0$ and $\sum_s a_{rs} \ne 0$ 
for some $r$). In particular, $u\mapsto\underline{V}(u)$ is injective. 
By the same reasoning, $\uu{V}'(u) \neq \uu{0}$. 
Indeed, \eqref{poiu} amounts to  
\begin{equation}\label{eq:am-v1}
V_r(u) - u \sum_{s=1}^q a_{rs} g_\alpha(V_s(u)) = 0 
\qquad \qquad 1 \le r \le q
\end{equation}
and differentiating this identity in $u$, we see that 
if $\uu{V}'(u)=\uu{0}$ then necessarily
\begin{equation}\label{eq:am-v2}
\sum_{s=1}^q a_{rs}g_\alpha(V_s(u))=0 \qquad \qquad 1 \le r \le q .
\end{equation}
Clearly, if (\ref{eq:am-v2}) holds then it follows
from (\ref{eq:am-v1}) that $\underline{V}(u)=\underline{0}$ 
and as we have already seen, for 
$\s \not\equiv 0$ it is then impossible for (\ref{eq:am-v2}) to hold.

Next we show that $\Vb \subseteq\C\times \C^q$ is a 
complex one-dimensional manifold, by finding for 
any point $u\in\Eaae$, a
suitable holomorphic mapping from a neighborhood $\Ub$
of $\uu{v}=(u,\uu{V}(u))$ in $\C^{q+1}$ 
to $\C^q$ having a Jacobian of rank $q$ at $\uu{v}$. 
Indeed, as it is not possible to have $V_1'(u)=\cdots=V'_q(u)=0$,   
we may assume without loss of generality that, for a given $u$, 
$V'_q(u)\neq 0$. Then, by the inverse function theorem 
there exists a neighborhood $U \subseteq \Eaae$ of $u$ 
with $V_q(\cdot)$ having an analytic inverse on the 
neighborhood $V_q(U)$ of $V_q(u)$. Thus, on 
the neighborhood 
$\Ub= U\times\C^{q-1}\times V_q(U)$ of $\uu{v}$ in $\C^{q+1}$
we have the holomorphic mapping $\uu{f}: \Ub \mapsto \C^q$ where 
$f_r(w,\uu{y})=y_r-V_r(V_q^{-1}(y_q))$ for $1\le r\le q-1$ and  
$f_q(w,\uu{y})=V_q(w)-y_q$. Clearly, $\uu{f}(w,\uu{y})=\uu{0}$ 
for $(w,\uu{y}) \in \Ub$ if and only if $\uu{y}=\uu{V}(w)$ and
$w \in U$, hence
$\{(w,\uu{y}) \in \Ub : \uu{f}(w,\uu{y}) = \uu{0}\}$ is 
precisely $\Vb \cap \Ub$. Further, since
$\partial_{y_r} f_s = \delta_{rs}$ for $1 \le r \le q-1$ 
and $\partial_{w} f_s = V_q'(w) \delta_{qs}$, the 
Jacobian determinant at $\uu{v}$ of $\uu{f}(\cdot,y_q)$ 
with $y_q$ fixed is $V_q'(u) \ne 0$. We conclude that  
rank$_{\uu{v}}(\uu{f})=q$ and    
$\Vb$ is a one dimensional complex manifold, as claimed.

Finally, while proving Proposition \ref{uniqueprop}
we found that 
${\rm det} [\partial_{\uu{y}}\uu{F}](0,\uu{0})=1$,
that $\uu{V}(u) \in (\KK_\a)^q$ 
for all $u \in \Eaa$ and that $\uu{V}(\cdot)$ 
is uniformly bounded on $\Ball (0,\varepsilon)$. 
With $g_\alpha(\cdot)$ uniformly bounded on $\KK_\alpha$
(and on compacts), it follows from \eqref{eq:am-v1} 
that $\|\uu{V}(u)\|_2 \le K|u|$ for some finite 
constant $K
=K(\sigma)$ and all $u \in \Eaae$. 
\end{proof}

\begin{rmk} The assumptions of Proposition \ref{prop:analytic}
do not yield a unique extension of $\underline{V}$ 
{\em around} boundary points of $\OO$. That is,
the extension provided there may well be non-analytic.
For example, the Cauchy-Stieltjes transform $y=G_2(z)$ 
of the semi-circle law $\mu_2$ at $z=u^{-1}$ is specified in 
terms of zeros of the holomorphic function 
$F(u,y) =y - u (y^2+1)$ on $\C^2$. 
It is not hard to check that for any positive 
$\varepsilon<1/2$ the unique analytic solution 
$y=V(u)$ of $F(u,y)=0$ on $\Ea_{1,\varepsilon}$ 
is then $V(u)=(1-\sqrt{1-4u^2})/(2u)$ for $u \ne 0$ and
$V(0)=0$. Following the arguments of Lemma \ref{lem:extension}, 
one finds that this injective function is uniformly bounded 
in the neighborhood of any boundary point 
of $\Ea_{1,\varepsilon}$ and its graph $\Vb$ is a one-dimensional 
manifold containing the origin (where 
$\partial F/\partial y = 1$). However, 
$V(x)$ does not have an analytic 
extension at $x=1/2$ as the corresponding density 
$\rho_2(t)$ is not real-analytic at $t = \pm 2$.
\end{rmk}

For the convenience of the reader, 
we summarize, following the reference \cite{chirka},
the terminology and results about analytic functions 
of several complex variables which we use in 
proving Proposition \ref{prop:analytic}.

A (local) analytic set is a subset $\Ab$ 
of a complex manifold $\Mb$ such that for any $\uu{a} \in \Ab$ 
there exists a neighborhood $\Ub$ of $\uu{a}$ in $\Mb$ 
and a holomorphic mapping $\uu{f} : \Ub \mapsto \C^n$
such that $\Ab \cap \Ub=\{\uu{z} \in \Ub \colon
\uu{f}(\uu{z}) = \uu{0} \}$ (in contrast with a manifold,
there is no condition on the rank of the Jacobian of the mapping $\uu{f}$).  
We call $\Ab \subseteq \Mb$ an analytic subset of the complex manifold 
$\Mb$ if this further applies at all $\uu{a} \in \Mb$ (and not only at
the points $\uu{a}$ in $\Ab$), and say that $\Ab$ is a {\em proper} 
analytic subset of $\Mb$ if $\Ab \neq \Mb$.
In particular, any embedded complex 
manifold is an analytic set (of $\C^q$), but, unless it 
is closed in $\C^q$, it cannot be an analytic subset of $\C^q$.
For example, $\Hb=\{ \uu{z} \in\C^q\colon \|\uu{z}\|_2 <1, z_1=0\}$ 
is a manifold (of dimension $q-1$), a (local) analytic set in $\C^q$, 
but not an analytic subset of $\C^q$. However, as observed in 
\cite[Section 1.2.1]{chirka}, every (local) analytic set on a 
complex manifold $\Mb$ is an analytic subset of a certain neighborhood of 
$\Mb$ (for example, $\Hb$  {\em is} an analytic subset of 
the open unit ball in $\C^q$).

A point of an analytic set $\Ab$ (on $\C^q$) 
is called {\em regular} if it has a neighborhood $\Ub$ 
(in $\C^q$) so that $\Ab \cap \Ub$ is a manifold in $\C^q$.
Clearly, the set ${\rm reg} \Ab$ of regular points of an
analytic set $\Ab$ is a union of manifolds (alternatively, 
an analytic set is a manifold around each 
of its regular points). Topologically,  
most points of an analytic set are regular. That is, 
for an arbitrary analytic set $\Ab$ the set 
${\rm reg} \Ab$ of regular points 
is everywhere dense in $\Ab$
(c.f. \cite[Section 1.2.3]{chirka}).
Thus, the dimension ${\rm dim}_{\uu{a}} \Ab$ of $\Ab$ 
{\em at a point} $\uu{a}\in \Ab$ is defined as the 
dimension of the manifold around $\uu{a}$ if 
$\uu{a} \in$reg $\Ab$ and in general by
$$
{\rm dim}_{\uu{a}} \Ab =\limsup_{\uu{z}\to \uu{a},\ \uu{z} \in{\rm reg}\Ab}
{\rm dim}_{\uu{z}} \Ab.
$$ 
The dimension of the analytic set $\Ab$, denoted ${\rm dim} \Ab$
is then the largest such number when $\uu{a}$ runs through $\Ab$
and an analytic set $\Ab$ is called $p$-dimensional if 
${\rm dim} \Ab = p$ (see \cite[Section 1.2.4]{chirka}).

An essential ingredient of our proof is the notion of
{\em irreducibility} and of irreducible components for 
analytic sets \cite[Section 1.5.3]{chirka}. An analytic subset $\Ab$ of
a complex manifold $\Mb$ is reducible in $\Mb$ 
if there exist two analytic
subsets $\Ab_1,\Ab_2$ of $\Mb$ 
so that $\Ab=\Ab_1\cup \Ab_2$ and $\Ab_1\neq \Ab
\neq \Ab_2$. Otherwise $\Ab$ is called irreducible.
For example,  $\Ab = \{\uu{z} \in\C^3\colon z_1z_2= z_1z_3=0\}$ is 
a reducible set, being the union of 
$\Ab_1= \{\uu{z} \in\C^3\colon z_2=z_3=0\}$ 
(a one dimensional manifold), and
$\Ab_2= \{ \uu{z}\in\C^3\colon z_1=0\}$
(a two dimensional manifold). 
An irreducible analytic subset $\Ab'$ of an analytic set $\Ab$ 
is called an {\em irreducible component} of $\Ab$ 
if every analytic subset $\Ab''$ of $\Ab$
such that $\Ab'\subsetneq \Ab''$ is reducible. 
It is known 
\cite[Theorem, Section 1.5.4]{chirka} that 
any analytic subset $\Ab$ of a complex manifold $\Mb$
has a unique decomposition into countably (or finitely) 
many irreducible components $\oo{\Sb}_j$, which are
the closures (in $\Mb$) of the partition $\{\Sb_j\}$ 
of ${\rm reg} \Ab$ into disjoint connected components.
Further, ${\rm dim} \oo{\Sb}_j = {\rm dim} \Sb_j$ 
\cite[Theorem, Section 1.5.1]{chirka} and
by definition of regular points 
each connected component $\Sb_j$ is a manifold 
(in case $\Mb=\C^q$). 

The importance of 
irreducibility for us resides in the following 'uniqueness' result: 
if $\Ab,\Ab'$ are analytic subsets of
a complex manifold, $\Ab$ is irreducible and $\Ab\not\subseteq \Ab'$, then
${\rm dim} \Ab\cap \Ab' < {\rm dim} \Ab$
\cite[Section 1.5.3, Corollary 1]{chirka}. 

Topological properties simplify considerably when  
a set $\Ab$ is contained in a proper analytic subset 
of a {\em connected} complex manifold $\Mb$. Indeed, in 
this context $\Mb \setminus \Ab$ is arc-wise connected 
and in case of a one dimensional manifold $\Mb$ 
(that is, a {\em Riemann surface}), we further have 
that $\Ab$ is locally finite i.e. $\Ab \cap \Db$ 
is a finite set for any compact $\Db \subseteq \Mb$  
(see \cite[Section 1.2.2]{chirka}).

\begin{proof}[Proof of Proposition \ref{prop:analytic}] 
Clearly, $\Vb$ is a connected set 
(being the graph of a continuous function on 
the connected set $\OO$). Further, 
by our assumptions, the connected one-dimensional 
complex manifold $\Vb$ is contained in the analytic subset 
$$
\Ab=\{(u,\uu{y})\in\C\times\C^q\colon
\uu{F}(u,\uu{y})=\uu{0}\},
$$
of $\C \times \C^q$ given by the zeros of the holomorphic 
mapping $\uu{F}$. 

We proceed to show the crux of our argument, that the 
closure $\oo{\Vb}$ of $\Vb$ (in $\C^{q+1}$) 
is part of a one-dimensional irreducible component of $\Ab$. 
To this end, consider the analytic subset
$$
\Abs  = \{(u,\uu{y})\in\C\times\C^q\colon
{\rm det}[\partial_{\uu{y}}\uu{F}](u,\uu{y})=0\}
$$
of $\C\times\C^q$. By definition, $\Vb \cap \Abs$ is an
analytic subset of $\Vb$, and it is a
proper subset, for we know that 
$\uu{v}_0 = (u_0,\uu{V}(u_0)) \in \Vb \setminus \Abs$. 
Further, by the implicit function theorem, the analytic subset $\Ab$ 
is regular at any $\uu{a} \in \Ab \setminus \Abs$, so 
$\Vb \setminus$ reg$\Ab$ is contained in the proper 
analytic subset $\Vb \cap \Abs$ of the Riemann surface $\Vb$. 
Consequently, by \cite[Proposition 2, Section 1.2.2]{chirka}, 
we deduce that $\Vb$ is an `almost regular' part of $\Ab$. 
That is, $\Vb \setminus$ reg$\Ab$ 
is locally finite and consequently the closure 
$\oo{\Vb}$ of $\Vb$ in $\C^{q+1}$ is the same as 
the closure of $\Vb \cap$reg $\Ab$. 
%
Further, by \cite[Proposition 3, Section 1.2.2]{chirka},
$\Vb \cap$reg $\Ab$ is arc-wise connected, hence included 
in {\em one} connected component $\Sb$ of reg$\Ab$,
with $\oo{\Vb}=\oo{\Vb \cap {\rm reg} \Ab}$ thus  
contained in the closure $\oo{\Sb}$ of $\Sb$ (in 
$\C^{q+1}$).

Recall that by definition the connected component 
$\Sb$ of reg$\Ab$ is a manifold and since 
$\uu{v}_0$ is in $\Vb \cap$ reg $\Ab$, 
we have that $\Sb$ contains the manifold 
$\Vb \cap \Ub$ for some neighborhood $\Ub$ 
of $\uu{v}_0$ (in $\C^{q+1}$). The connected
manifold $\Vb \cap \Ub$ has an accumulation 
point in both $\Vb$ and $\Sb$, so all three
have the same dimension, that is,
${\rm dim} \Sb = {\rm dim} \Vb = 1$ (see \cite[Section 1.2.2]{chirka}).
As shown in \cite[Theorem, Section 1.5.4]{chirka}, 
the irreducible component of $\Ab$ containing
$\Sb$ is its closure $\oo{\Sb}$, which  
by the definition of an irreducible component
is an analytic subset of $\C^{q+1}$ and
further by \cite[Theorem, Section 1.5.1]{chirka}
${\rm dim} \oo{\Sb} = {\rm dim} \Sb =1$. 

We claim that if $\uu{V}(u)$ is uniformly bounded on 
$\OO \cap U$ for some neighborhood $U$ (in $\C$) of 
a boundary point $x \in \oo{\OO}$ 
where $\OO$ is locally connected, 
then the existence of the one-dimensional irreducible 
analytic subset $\oo{\Sb}$ of $\C^{q+1}$ insures that 
$\uu{V}(\cdot)$ extends continuously at $x$ 
such that $\uu{F}(x,\uu{V}(x))=\uu{0}$.
Indeed, since $\uu{V}$ is uniformly bounded on
$\OO\cap U$, by the continuity of $\uu{V}(\cdot)$ on 
$\OO$ the cluster set $Cl(x)
$ 
of all limit points of $\{ \uu{V}(u) : u \in \OO\}$ 
as $u \to x$, is a non-empty, compact, {\em connected} 
subset of $\C^q$ (see the proof
given in \cite[Theorem 1.1]{collin} for $q=1$). 
Clearly, $\{x\}\times Cl(x)$ is contained in the analytic subset 
$\Ab(x) = \{ (u,\uu{y}) \in \Ab : u = x \}$ of $\C \times \C^q$ 
as well as in the closure $\oo{\Vb}$ of $\Vb$ (in $\C^{q+1}$).
With $\oo{\Vb} \subseteq \oo{\Sb}$, we thus deduce that
$\{x\}\times Cl(x) \subseteq \Ab(x) \cap \oo{\Sb}$.

Recall that $\oo{\Sb}$ is a one-dimensional, irreducible 
analytic subset of $\C^{q+1}$. Since 
$\oo{\Sb} \not\subseteq \Ab(x)$ (as $\uu{v}_0 \in \Vb$ and
$u_0$ is not a boundary point of $\OO$), by
\cite[Corollary 1, Section 1.5.3]{chirka} 
we have that ${\rm dim} \Ab(x)\cap\oo{\Sb} =0$.
Thus, $\Ab(x) \cap \oo{\Sb}$ is a discrete (analytic) set, so 
its connected subset $\{x\}\times Cl(x)$ must be a single point,
i.e. $\uu{V}$ extends continuously at $x$. 
Moreover, $\Ab$ is a a closed subset of $\C \times \C^q$ 
(by continuity of $\uu{F}$), hence the extension $\uu{V}(x)$ 
of $\uu{V}$ satisfies $\uu{F}(x,\uu{V}(x))=\uu{0}$, as claimed.
\end{proof}

\subsection{Limiting spectral measures: proof of 
Theorem \ref{theo-limitpoint-uniq-amir}}

Fixing $z \in \C^+$ let $(\mu^z_s, 1 \leq s \leq q)$ 
denote some limit point of the 
compactly supported $(\E[L_{N,s}^{z,\kappa}], 1 \leq s \leq q)$.
Then, on the corresponding subsequence,
$X_{N,r}(z)$ converges for $r=1,\ldots,q$
to $\int x^{\frac{\alpha}{2}} d\mu_r^z$ 
which by Propositions \ref{projectionoflimitpoint}
and \ref{uniqueprop} thus coincides with the 
unique analytic solution $X_r(z)$ of \eqref{cocottes}  
bounded by $(\Im(z))^{-\frac{\a}{2}}$.
By \eqref{limitpointeq} 
(for $f(x)=x$ bounded and continuous on $\Db (z)$), 
the identity $z^{-1} = -i \int_0^\infty e^{i t z} dt$, and
Fubini's theorem, we deduce that for each $r \in \{1,\ldots, q\}$,
\begin{eqnarray}\label{eq:newam2}
\int xd\mu^z_r &=&
- i \int_0^\infty e^{itz}\prod_{s=1}^q \Big(\int  
\exp\{ -it \sigma_{rs}^2  \Delta_s^{\frac{2}{\alpha}} x_s \}
dP^{\mu^z_s}(x_s) \, \Big) dt\nonumber \\
&=&
-i \int_0^\infty e^{itz} \exp\{-(it)^{\frac{\alpha}{2}} \wh{X}_r (z) \} dt 
= z^{-1} g_{\alpha,2} (Y_r(z)) \,,
\end{eqnarray}
where we get the latter equality from \eqref{cocott} and the 
definition \eqref{eq:yrdef} of $\wh{X}_r(z)$, followed by the application of 
\eqref{eq:gziden} with $\beta=2$.
In particular, 
by Proposition \ref{uniqueprop}
$\int x d\mu^z_r$ are uniquely determined (for all $r$ and $z\in\C^+$), 
hence $\E[L_{N,s}^{z,\kappa}(x)]$ converges
as $N \to \infty$ to the right side of \eqref{eq:newam2}.

Next, by Lemma \ref{tight} the sequence $\E[\mun_{\bA_N^{\kappa,\s}}]$
is tight for the topology of weak convergence. Further, recall that
for any $z \in \C^+$ and all $N$,
$$
\int \frac{1}{z-x} d\mun_{\bA_N^{\kappa,\sigma}}(x)
=\sum_{s=1}^q\De_{N,s} L_{N,s}^{z,\kappa}(x).
$$
Hence, any limit point $\mu^\sigma$ of $\E[\mun_{\bA_N^{\kappa,\s}}]$
is such that for each $z \in \C^+$,
\begin{equation}\label{eq:newam1}
\int \frac{1}{z-x} d\mu^{\sigma}(x)
=\sum_{s=1}^q\De_s \int x d\mu^z_s(x) \,.
\end{equation}
Recall that $g_{\alpha,2}=h_\alpha$, so 
combining \eqref{eq:newam2} and \eqref{eq:newam1} we thus 
arrive at the stated formula \eqref{eqG} for the values of the
Cauchy-Stieltjes transform $G_{\alpha,\sigma}(z)$ 
of the probability measure $\mu^{\sigma}$ 
on the real line, at all $z\in \C^+$. Since 
$h_\a$ is uniformly bounded on the closed
set $\KK_\a$ (see Lemma \ref{boundg}), and 
$Y_s(z) \in \KK_\a$ for all $z \in \C^+$
and $1 \le s \le q$,
 we deduce from \eqref{eqG} 
that $G_{\a,\s}(z)$ is uniformly 
bounded on $\C^+ \cap \Ball (0,\delta)^c$ for each $\delta>0$. 
By the Stieltjes-Perron inversion formula, it follows that
the density $\rho^\s$ of the probability measure $\mu^\s$ 
with respect to Lebesgue measure on $\R \setminus \{0\}$
is bounded on $(-\delta,\delta)^c$ for any $\delta>0$.

With $G_{\alpha,\sigma}(z)$ uniquely 
determined, we conclude that so is the weak limit 
$\mu^\sigma$ of $\E[\mun_{\bA_N^{\kappa,\sigma}}]$.
Further, applying Lemma \ref{concentration}
for $f(x)=x$ and considering the union bound 
over $1 \le s \le q$, we find that, 
with $\e= 1 - \kappa(2-\alpha) >0$, 
for some $c(z)$ finite on $\C^+$, 
any $z \in \C^+$, $\de>0$ and $N \in \N$,
$$
\P\Big(\big| 
\int \frac{1}{z-x} d\mun_{\bA_N^{\kappa,\sigma}}(x)
- 
\E [\int \frac{1}{z-x} d\mun_{\bA_N^{\kappa,\sigma}}(x)]
\big|\ge \d\Big)
\le \frac{q c(z)}{\d^2} N^{-\e} \,.
$$
Consequently, setting $\phi(n)=[n^\gamma]$ for $\gamma=2/\e$, 
by the Borel-Cantelli lemma, with probability one, as $n \to \infty$, 
$$
G_n(z) := \int \frac{1}{z-x} d\mun_{\bA_{\phi(n)}^{\kappa,\sigma}}(x) \to 
G_{\a,\s}(z) \,. 
$$
Since $G_n(z) \le (\Im(z))^{-1}$ for all $n$ and
$z \in \C^+$, 
applying this for a countable collection $z_k$ with 
a cluster point in $\C^+$ we deduce by Vitali's convergence theorem
that with probability one, $G_n(z) \to G_{\a,\s}(z)$ for all $z \in \C^+$.
Such convergence of the 
Cauchy-Stieltjes transforms 
implies of course
that $\mun_{\bA_{\phi(n)}^{\kappa,\sigma}}$ converges weakly to 
$\mu^\s$ and by \eqref{trunc2} we deduce after yet another 
application of the Borel-Cantelli lemma, that with probability one 
$\mun_{\bA_{\phi(n)}^\sigma}$ converges weakly to $\mu^\s$. 
Finally, since $\phi(n-1)/\phi(n) \to 1$ we have from Lemma 
\ref{lem:interp} that the same weak convergence to $\mu^\s$ 
holds for $\mun_{\bA^\s_N}$.

With $h_\a(\oo{y})=\oo{h_\a}(y)$, 
combining the identities $Y_r(-\oo{z})=\oo{Y_r}(z)$ of 
Proposition \ref{uniqueprop}
with the formula \eqref{eqG}
for the Cauchy-Stieltjes transform 
$G_{\alpha,\sigma}$ of the probability measure $\mu^\s$ on
$\R$ we find that  
$G_{\alpha,\sigma}(-\oo{z}) 
= -\oo{G}_{\alpha,\sigma}(z)
=-G_{\alpha,\sigma}(\oo{z})$ for all $z \in \C^+$, hence
necessarily $\mu^\s(\cdot)=\mu^\s(-\cdot)$ is symmetric about 
zero. Further, as shown in Proposition \ref{uniqueprop}, 
$z^\alpha \uu{Y} (z)$ is uniformly bounded and
extends analytically through the subset $(R,\infty)$, where 
$\uu{Y}(z)=\uu{V}(z^{-\alpha}) \in ({\mathcal K}_\alpha)^q$
is the unique analytic solution of \eqref{systemeqY} on $z\in \C^+$ 
that tend to zero as $|z| \to \infty$ (and as shown in Lemma 
\ref{lem:extension}
$z\mapsto \uu{Y}(z)$ is injective 
when $\sigma\not\equiv 0$). 
If $\s \equiv 0$ then 
$\uu{V}(u)=\uu{0}$ is analytic on $\C$.
Turning to $\s \not\equiv 0$, in view of Lemma 
\ref{lem:extension} the function 
$\uu{V}$ is uniformly bounded on $\Eaae \cap \Db$ for 
any compact $\Db \subseteq \C$. Thus, combining 
Lemma \ref{lem:extension} with Proposition \ref{prop:analytic} 
we find that $\uu{V}(u)$ has a continuous, algebraic extension to
$(0,\infty)$.
As $Y_r(-\overline{z})=\overline{Y_r}(z)$,
this yields the continuous, algebraic extension of $\uu{Y}(z)$ to 
$\reals \setminus \{0\}$, analytic on $(R,\infty)$, 
from which we get 
by \eqref{eqG} and the analyticity of $h_\alpha(\cdot)$
the corresponding continuous/algebraic/analytic extension 
of $G_{\alpha,\sigma}(z)$. 
Recall Plemelj formula, that for $x \ne 0$, 
the limit as $\epsilon \downarrow 0$ 
of $-\pi^{-1} \Im(G_{\alpha,\s} (x+i\epsilon))$ is then 
precisely the continuous density $\rho^\s(x)$ 
of $\mu^\s$ with respect to Lebesgue measure on $\R \setminus \{0\}$,
and $\rho^\s(x)$ is real-analytic on $(R,\infty)$.

\section{Proof of Theorem \ref{weakening}}\label{sec:gen}

We start with the following consequence of 
Proposition \ref{prop-dense}
and Theorem \ref{theo-limitpoint-uniq-amir}.
\begin{cor}\label{cor-sym}
For any  $\sigma \in \FF_\alpha$, the probability 
measures $\E[\mun_{\bA_N^\sigma}]$ converge weakly towards 
some symmetric probability measure
$\mu^{\sigma}$.
\end{cor}
\begin{proof} We approximate $\sigma$
in $L^2_\star([0,1]^2)$ 
by a sequence of piecewise constant functions $\sigma_p$.
Applying Theorem \ref{theo-limitpoint-uniq-amir}
for $\sigma=\sigma_p$ we deduce that
hypothesis \eqref{hyp1} holds. Hence,
by Proposition \ref{prop-dense}
$\E[\mun_{\bA_N^\sigma}]$ converges weakly towards 
the limit $\mu^\sigma$ 
of the corresponding measures $\mu^{\sigma_p}$. We have
seen already that $\mu^{\sigma_p}$ are symmetric 
measures, hence so is their limit $\mu^\s$.
\end{proof}

Fixing $\s \in \FF_\alpha$ we proceed to 
characterize the limiting measure $\mu^\s$.
To this end, recall that $k_\s :=  \| \, |\s|^\alpha \|$ 
is finite and fix a sequence $\s_p\in\PC$ 
that converges to $\s$ in $L^2_\star$, 
satisfying \eqref{domaine} and such that 
$\sup_{p\in\N} \| \, |\s_p|^\alpha \| \leq 2 k_\s$.
For each 
$p \in \N$ let $0=b_0^p< b_1^p < \cdots < b^p_{q(\s_p)} =1$ 
denote the finite partition of $[0,1]$ 
induced by $\s_p$ and per $z \in \C^+$ consider the piecewise constant
function $Y^{\sigma_p}_.(z) : (0,1] \to \KK_\alpha$
such that 
$$
Y^{\sigma_p}_x(z)=Y_s(z)\mbox{ for } x\in (b_{s-1}^p,b_s^p]
\mbox{ and } s=1,\ldots,q(\s_p) \,,
$$
where $Y_s(z)\in \KK_\alpha$ is the unique
collection of (analytic) functions of $z \in \C^+$ 
that satisfy \eqref{systemeqY} for the $q \ts q$ matrix 
of entries $\s_{rs}:=\s_p(b_r^p,b_s^p)$, as in 
Theorem \ref{theo-limitpoint-uniq-amir}. This way 
\eqref{systeqlim} holds for $\s=\s_p$ and each $p \in \N$ 
(being precisely \eqref{systemeqY}). 

We next show the existence of $R=R(\s)$ finite
such that if $|z| \geq R$ then $(Y^{\sigma_p}_.(z), p\in\N)$
is a Cauchy sequence for the $L^\infty$-norm.
To this end, it is convenient to 
view \eqref{systeqlim} (at each $z \in \C^+$)
as the fixed point equation 
in $L^\infty((0,1];{\mathcal K}_\alpha)$ 
\begin{equation}\label{systeqalt}
Y^{\sigma}_. = F_z(\s,Y^\s_.), \quad 
F_z(\s,Y) := C_\a z^{-\alpha} \int_0^1 |\sigma(\cdot,v)|^\alpha
g_\alpha(Y_v ) dv \,.
\end{equation}
Then, with 
$\|g_\alpha\|_{\KK_\alpha} := \sup \{ |g_\alpha(y)| : y \in 
\KK_\alpha\}$ finite by Lemma \ref{boundg}, bounding the
$L^\infty$-norm of $F_z(\s,Y)$ for $Y \in \KK_\alpha$ 
we deduce from \eqref{systeqalt} that for any $\epsilon>0$
\begin{equation}\label{unifbY}
\sup_{|z| \geq \epsilon} \, \sup_{p\in \N} \, \| Y^{\sigma_p}_.\|_\infty \le  
2 k_\s \epsilon^{-\alpha} |C_\alpha| \|g_\alpha\|_{\mathcal K_\alpha} =: r_\s 
\end{equation}
is finite. Note that for   
$\|\wt{Y}\|_\infty \leq r$, $\|\wh{Y}\|_\infty \leq r$ and   
measurable $\wt{\s}(\cdot,\cdot)$, $\wh{\s}(\cdot,\cdot)$, 
\begin{equation}\label{eq:Fbd}
\| F_z(\wt{\s},\wt{Y})-F_z(\wh{\s},\wh{Y}) \|_\infty
\le |z|^{-\alpha} \|g_\alpha\|_r 
\Big[ \, \| \, |\wt{\s}|^\alpha - |\wh{\s}|^\alpha \|
+ \| \, |\wt{\s}|^\alpha \| \, \| \wt{Y}-\wh{Y}\|_\infty \Big] \,,
\end{equation}
where $\|g_\alpha\|_r$ is the sum of the supremum and 
Lipschitz norms of $y \mapsto C_\alpha g_\alpha(y)$ 
on the ball $\{y \in \C : |y| \leq r\}$.  
Suppressing hereafter the dependence of $Y_x^{\s_p}(z)$ on $z$,
since $(\s_p,Y^{\s_p}_.)$, $p \in \N$, 
satisfy \eqref{systeqalt}, from \eqref{unifbY} and
\eqref{eq:Fbd} we have that for any $p,q \in \N$ and $|z| \geq \epsilon$,
\begin{eqnarray*}
\| Y^{\sigma_q}_.-Y^{\sigma_p}_.\|_\infty \le
|z|^{-\alpha} \|g_\alpha\|_{r_\s} 
\Big[ \| \, |\s_q|^\alpha - |\s_p|^\alpha\| + 2 k_\s   
\| Y^{\sigma_q}_.-Y^{\sigma_p}_.\|_\infty \Big] \,.
\end{eqnarray*}
Taking $R=R(\s) \ge \epsilon$ finite such that 
$R^{-\alpha} \|g_\alpha\|_{r_\s} k_\s \leq 1/3$, 
this implies that for $|z| \geq R$
$$
\| Y^{\sigma_q}_. -Y^{\sigma_p}_. \|_\infty \le 
3 |z|^{-\alpha} \|g_\alpha\|_{r_\s}
 \| \, |\s_q|^\alpha - |\s_p|^\alpha\| \,.
$$
In view of \eqref{domaine}, we conclude that 
$(Y^{\sigma_p}_.,p\in \N)$ is for each $|z| \geq R$ 
a Cauchy sequence
in $L^\infty(0,1])$, which thus converges in this space to a
bounded measurable function $Y^{\s}_.$ from $(0,1]$ to 
the closed set ${\mathcal K}_\alpha$. Further, then
$\|Y^\s_.\|_\infty \leq r_\s$ (see \eqref{unifbY}), 
so from \eqref{eq:Fbd} and \eqref{domaine} we deduce that 
$$
\| F_z(\s,Y^\s_.)-F_z(\s_p,Y^{\s_p}_.) \|_\infty
\le \epsilon^{-\alpha} \|g_\alpha\|_{r_\s} \big[ \, \|
\, |\s|^\alpha - |\s_p|^\alpha\|
+ k_\s \| Y^\s_.-Y^{\s_p}_.\|_\infty \big] \to 0,
$$
as $p \to \infty$. With \eqref{systeqalt}
holding for the pairs $(\s_p,Y^{\s_p}_.)$, $p \in \N$,
it follows that the same applies for $(\s,Y^\s_.)$, 
thus establishing \eqref{systeqlim}. 

Turning to show the uniqueness of the solution to  
\eqref{systeqlim}, suppose $Y_j=F_z(\s,Y_j)$ 
for $\s(\cdot,\cdot)$ such that 
$k_\s = \| \, |\s|^\alpha \|$ is finite,
some $|z| \geq R(\s)$ and   
measurable $Y_j : (0,1] \to {\mathcal K}_\alpha$, $j=1,2$. 
Then, as in the derivation of \eqref{unifbY} we have that
$\|Y_j\|_\infty \leq r_\s$ for $j=1,2$. So, applying 
\eqref{eq:Fbd} once more, 
$$
\| Y_1 - Y_2 \|_\infty = \| F_z(\s,Y_1)-F_z(\s,Y_2) \|_\infty
\le |z|^{-\alpha} \|g_\alpha\|_{r_\s} 
k_\s \| Y_1-Y_2\|_\infty \le \frac{1}{3} \| Y_1 - Y_2 \|_\infty
$$
and necessarily $Y_1=Y_2$ almost everywhere on $(0,1]$.

To recap, the sequence of holomorphic mappings 
$Y^{\s_p}$ from $\C^+$ to the closed subset
$\Fb :=L^\infty((0,1];{\mathcal K}_\alpha)$ of the
complex Banach space $L^\infty((0,1];\C)$ is such that 
$Y^{\s_p}(z) \to Y^{\s}(z)$ in $\Fb$ at each point $z$ of 
the non-empty open subset $\Ball (0,R)^c\cap \C^+$. Further,
in view of \eqref{unifbY} we have that 
$(Y^{\s_p}, p \in \N)$ is locally uniformly 
bounded on $\C^+$,
hence by  
Vitali's convergence theorem for vector-valued holomorphic
mappings, it converges at every $z \in \C^+$ to 
an analytic mapping $Y^\s : \C^+ \mapsto \Fb$ 
(see \cite[Theorem 14.16]{Chae}). We also
characterized $Y^\s(z)$ for each $|z| \ge R$ as
the unique solution in $\Fb$ of \eqref{systeqlim}, so
by the identity theorem for vector-valued holomorphic 
mappings (see \cite[Exercise 9C]{Chae}), we have thus
uniquely determined $Y^\s : \C^+ \mapsto \Fb$.

Next, note that the identity \eqref{eqGlim} holds 
for $\s=\s_p \in \PC$, $p \in \N$, in which case it 
is merely the formula \eqref{eqG}. 
Recall Proposition \ref{prop-dense} 
that due to the $L^2_\star$-convergence of $\s_p$ to $\s$, 
for each $z \in \C^+$
the left hand side of these identities converge 
as $p \to \infty$ to $G_{\a,\s}(z) := \int (z-x)^{-1} d\mu^\s (x)$.
If in addition $|z| \geq R(\s)$ then 
$\| Y^{\s_p}_. - Y^\s_.\|_\infty \to 0$ and 
by dominated convergence the right hand sides of 
same identities converge to the corresponding 
expression for $Y^\s_.(z)$. Thus, 
\eqref{eqGlim} holds also for $\s \in \FF_\alpha$ and $|z| \ge R(\s)$.
With $\mu^\s$ a probability measure on $\R$, the left side of 
\eqref{eqGlim} is obviously an analytic function of $z \in \C^+$.
Further, the entire function $h_\alpha(\cdot)$ and 
its first two derivatives are uniformly bounded on 
the set
${\mathcal K}_\alpha$
(see Lemma \ref{boundg}), 
in which the analytic mapping $Y^\s : \C^+ \mapsto \Fb$ takes values. 
Hence, 
it is not hard to see that
$z \mapsto \int_0^1 h_\alpha(Y^\s_v(z))dv$ is also analytic on $\C^+$. 
We thus deduce by the identity theorem that   
\eqref{eqGlim} holds for all $z \in \C^+$.
Consequently, with $\int_0^1 h_\alpha(Y^\s_v(z))dv$   
uniformly bounded on $\C^+$, 
the Cauchy-Stieltjes 
transform of $\mu^\s$ is 
uniformly bounded on $\C^+ \cap \Ball (0,\delta)^c$. This in 
turn implies 
(by the Stieltjes-Perron inversion formula), 
that the density $\rho^\s$ of $\mu^\s$ 
with respect to Lebesgue measure on $\R \setminus \{0\}$
is bounded outside any neighborhood of zero.

We have seen already that $\|Y^\s\|_\infty \le c(\s) |z|^{-\a}$
for some $c(\s)$ finite and all $|z| \ge R$.
Hence, for $z \in \C^+$ such that $|z| \ge R$,
we have from \eqref{eqGlim} and \eqref{systeqlim}
that 
\begin{eqnarray*}
G_{\a,\s}(z) &=& 
\frac{1}{z} \big[ h_\a(0) + h_\a'(0) \int_0^1 Y_x^\s(z) dx 
+ O(|z|^{-2\a}) \big] \\
&=& 
\frac{1}{z} \big[ h_\a(0) + z^{-\a} C_\a h_\a'(0) g_\a(0) 
\int_0^1 \int_0^1 |\s(x,v)|^\a dx dv 
+ O(|z|^{-2\a}) \big] \,.
\end{eqnarray*} 
Recall Plemelj formula, that $\rho^\s(t)$ 
is the limit of
$-\pi^{-1} \Im\big(G_{\a,\s}(t+i\epsilon)\big)$ 
as $\epsilon \downarrow 0$. 
Thus, as $h_\a(0) \in \R$, it follows  
that $t^{\a+1}\rho^\s(t) \to L_\a
\int |\s(x,v)|^\a dx dv$ as $t \to \infty$
and it is not hard to check that 
$L_\a= -\pi^{-1} h_\a'(0) g_\a(0) \Im(C_\a)$
equals $\frac{\alpha}{2}$ (by Euler's reflection formula
for the Gamma function). 

Turning to verify the last statement of the theorem, note that
the equivalence between $\s \in \FF_\alpha$ and 
$\wt{\s} \in \PC$ implies that the piecewise constant 
$Y^{\wt{\s}}_. (z) : (0,1] \mapsto {\mathcal K}_\alpha$
we have constructed before out of $(Y_s(z), 1 \leq s \leq q)$ 
satisfies \eqref{systeqlim} for any $x \in (0,1]$ 
and all $z \in \C^+$. It then follows by the 
uniqueness of such solution of \eqref{systeqlim} that 
$Y^{\wt{\s}}_x (z) = Y^\s_x (z)$ for all $z \in \C^+$ such that 
$|z| \geq R(\s)$ and almost every $x \in (0,1]$. In view of 
\eqref{eqGlim},
 the Cauchy-Stieltjes transform of $\mu^\s$ 
coincides for such $z$ with the Cauchy-Stieltjes
transform $G_{\alpha,\wt{\sigma}}(z)$ of $\mu^{\wt{\s}}$. 
As such information uniquely 
determines the probability measure in question, it follows
that $\mu^\s = \mu^{\wt{\s}}$.

\section{Proof of Proposition \ref{gammaone} and Theorem \ref{wishart-amir}}
\label{sec:wis}

\subsection{Convergence to $\mu_\a^\gamma$ and its characterization}
\label{sec:wis-conv}
Consider the $(N+M)$-dimensional square matrix 
$$
\bA_{N,M}=\left(
\begin{array}{cc}
0&a_{N+M}^{-1} \bX_{N,M}\cr
a_{N+M}^{-1} \bX_{N,M}^t&0\cr
\end{array}
\right),$$
noting that $\bB_{N,M}=  \bA_{N,M}^2$ is then of the form
$$
\bB_{N,M}= \left(
\begin{array}{cc}
a_{N+M}^{-2} \bX_{N,M} \bX_{N,M}^t&0\cr
0&a_{N+M}^{-2} \bX_{N,M}^t\bX_{N,M}\cr
\end{array}
\right)=: \left(
\begin{array}{cc} \bW_{N,M}& 0\cr
0&  \widetilde \bW_{N,M}\cr
\end{array}
\right)
$$
and that the eigenvalues of $\bW_{N,M}$ consist of 
the $M$ eigenvalues of $\widetilde \bW_{N,M}$ 
augmented by 
$N-M$ zero eigenvalues.
Therefore,
\begin{equation}\label{convwis}
\mun_{\bB_{N,M}}=\frac{2 N}{N+M}
 \mun_{ \bW_{N,M}} +\frac{M-N}{N+M}\d_0.
\end{equation}
We next show that with probability one $\mun_{\bB_{N,M}}$ 
converges weakly. Since $\bB_{N,M}=\bA_{N,M}^2$,
for any $f(\cdot)$ bounded and continuous,
 \begin{equation}\label{convwis2}
\int f(x)d\mun_{\bB_{N,M}}= \int f(x^2) d\mun_{\bA_{N,M}}
\end{equation} 
so that it is enough to prove the convergence of $\mun_{\bA_{N,M}}$.
To this end, consider $\bA_{N+M}^\sigma$ for  
\begin{equation}\label{eq:covsig}
\sigma(x,y)=\left\lbc
\begin{array}{l}
1 \;\; \mbox{ if }\,  x, y\in (\frac{1}{1+\gamma},1]\ts
(0, \frac{1}{1+\gamma}]\bigcup (0, \frac{1}{1+\gamma}]\ts 
 (\frac{1}{1+\gamma},1]\cr
0 \qquad \qquad \mbox{ otherwise.}
\end{array}\right.
\end{equation}
Note that with $M/N \to \gamma$ and  
$$
\mbox{rank}(\bA_{N,M}-\bA_{N+M}^\sigma) \leq 2 \Big| \big[ \frac{N+M}{1+\gamma} 
\big] - N \Big| \,, 
$$ 
it follows by Lidskii's theorem that 
$d_1(\mun_{\bA_{N,M}},\mun_{\bA_{N+M}^\sigma}) \to 0$ as $N \to \infty$.
Therefore, applying Theorem \ref{theo-limitpoint-uniq-amir} 
we deduce that with probability one 
$\mun_{\bA_{N,M}}$ converges weakly to the non-random probability measure
$\mu^\sigma$. By \eqref{convwis2}
and  \eqref{convwis} this implies that $\mun_{\bB_{N,M}}$
and $\mun_{\bW_{N,M}}$ also 
converge weakly to non-random probability measures,
\begin{equation}\label{eq:alicef1}
\mu_B:=\frac{2}{1+\gamma} \mu^\gamma_\alpha +\frac{\gamma-1}{\gamma+1}\d_0 
\end{equation}
and $\mu^\gamma_\alpha$, respectively. 

We proceed to show that 
for $z \in \C^+$  the Cauchy-Stieltjes transform of $\mu^\gamma_\alpha$ 
is 
\begin{equation}
\label{defGwis}
G_\alpha^\gamma(z)
=\frac{1}{z} h_\alpha(Y_1(\sqrt{z}\,)) = \frac{1-\gamma}{z}+\frac{\gamma}{z}
h_\alpha(Y_2(\sqrt{z}\,)) \,.
\end{equation}
Indeed, note that $(Y_1(z), Y_2(z))$ of \eqref{defywis} are precisely  
the solution of \eqref{systemeqY} considered in Proposition \ref{uniqueprop}
for $z \in \C^+$ and our special choice of $\sigma(\cdot,\cdot)$. 
Theorem \ref{theo-limitpoint-uniq-amir} thus asserts that 
the Cauchy-Stieltjes transform 
$G_{\alpha,\sigma}$ of $\mu^\sigma$ is then such that, 
for any $z \in \C^+$,  
\begin{equation}\label{eq:amir-new2}
z G_{\alpha,\sigma}(z)=
\frac{1}{1+\gamma} h_\alpha(Y_1(z)) +  
 \frac{\gamma}{1+\gamma}  h_\alpha(Y_2(z)) \,.
\end{equation}
Moreover, by 
\eqref{convwis2} and the symmetry of the law $\mu^\sigma$ (see 
Corollary \ref{cor-sym}), we have that 
$$
\int\frac{1}{z-x}d\mu_B(x)= \int\frac{1}{z-x^2} d\mu^\sigma(x)
= \frac{1}{\sqrt{z}} \int\frac{1}{\sqrt{z}-x}  d\mu^\sigma(x)
\,.
$$ From this and the formula \eqref{eq:alicef1}
relating $\mu_B$ to $\mu^\gamma_\alpha$, we deduce that
\begin{equation}\label{eq:amir-new1}
G_\alpha^\gamma(z)
=\frac{1+\gamma}{2\sqrt{z}} G_{\alpha,\sigma}(\sqrt{z}\,)
+\frac{1-\gamma}{2 z} \,.
\end{equation}
Multiplying the left identity of \eqref{defywis} by $Y_2(z)$
and the right identity of \eqref{defywis} by $Y_1(z)$ we find that 
$Y_1(z) g_\alpha(Y_1(z)) = \gamma Y_2(z) g_\alpha(Y_2(z))$
and hence
\begin{equation}\label{eq:serban}
h_\alpha(Y_1(z))=1-\gamma + \gamma h_\alpha(Y_2(z)) \,.
\end{equation}
Upon combining (\ref{eq:amir-new2}), (\ref{eq:amir-new1})
and (\ref{eq:serban}) we get the formula \eqref{defGwis}. 

\subsection{Analysis of the limiting measures}\label{sec:anal-lim}
In case $\gamma=1$, the function  
$\sigma(x,y)$ of \eqref{eq:covsig} is equivalent to 
the constant $\wt{\sigma}=2^{-1/\alpha}$, which as 
in Remark \ref{rmkbnd} implies that $\mu^\s$ has
the density $\rho^\s(t)= 2^{1/\alpha} \rho_\a(2^{1/\alpha} t)$. 
Further, 
we see from \eqref{eq:amir-new1} that 
$G_\a^1(z)=G_{\a,\s}(\sqrt{z}\,)/\sqrt{z}$, 
so the probability 
measure $\mu_\a^1$ on $(0,\infty)$ has the density
$\rho^\s(\sqrt{t})/\sqrt{t}$, as stated. 

Considering hereafter $\gamma \in (0,1)$, observe 
that 
by Theorem 
\ref{theo-limitpoint-uniq-amir}, $Y_1(z)$ and $Y_2(z)$ extend 
continuously to functions on $(0,\infty)$ that are  
analytic outside of some bounded set.
By
the analyticity of $h_\alpha(\cdot)$ 
and \eqref{defGwis} we have the corresponding 
continuous extension of $G_\alpha^\gamma(z)$, 
whereby Plemelj formula
provides the density 
$\rho^\gamma_\a(t)=-\pi^{-1} \Im(G_\alpha^\gamma(t))$ 
of $\mu^\gamma_\alpha$ with respect to Lebesgue measure,
as in \eqref{eq:den-wis}. In particular,  
$\rho_\a^\gamma(t)=\frac{1+\gamma}{2\sqrt{t}}\rho^\s(\sqrt{t}\,)$ 
by \eqref{eq:amir-new1}, 
with $\s(\cdot,\cdot)$ of \eqref{eq:covsig}, 
so we read the tail behavior of 
$\rho_\a^\gamma$ out of that of $\rho^\s$ (per Theorem 
\ref{weakening}).

Turning next to the behavior near zero of the probability measure 
$\mu_\alpha^\gamma$, recall that $G_\alpha^\gamma(z)$ is analytic
outside the support $[0,\infty)$ of $\mu_\alpha^\gamma$ 
and the non-tangential limit of 
$z G_\alpha^\gamma(z)$ at the boundary point $z=0$ 
(i.e., its limit as $|z| \to 0$ while 
$\theta_0 \le \arg(z) \le 2\pi-\theta_0$ 
for some fixed $\theta_0>0$), exists 
and equals to the mass at zero of this measure. 
Further, the identity \eqref{defGwis} extends by continuity 
to $z=-x^2$, $x>0$ and $\sqrt{z}=ix \in \C^+$, hence 
\begin{equation}\label{eq:nontgh}
\mu_\alpha^\gamma(\{0\})=
\lim_{\stackrel{z \to 0}{\sphericalangle}} z G_\alpha^\gamma(z)=
\lim_{x\downarrow 0} h_\alpha (Y_1(ix\,)) = 
1-\gamma+\gamma\lim_{x \downarrow 0} h_\alpha(Y_2(ix\,))  \,.
\end{equation}
Since $Y_s(-\oo{z})=\oo{Y_s}(z)$ for $s=1,2$ and all $z \in \C^+$
(see Proposition \ref{uniqueprop}), 
we have in particular that $Y_1(ix)$ and $Y_2(ix)$ 
are real-valued for all $x>0$. As 
$g_\alpha(y)>0$ for $y \in \R$,
it further follows from \eqref{defywis} 
that $Y_s(i\R^+) \subseteq \R^+$ for $s=1,2$. 
With $h_\alpha: \R^+ \to \R^+$ monotone decreasing and
$h_\alpha(y) \to 0$ as $\Re(y) \to \infty$, it thus follows
from \eqref{eq:serban} that 
$h_\alpha(Y_1(ix)) \ge 1-\gamma$ for all $x>0$ 
and consequently, that 
$(Y_1(ix),x>0)$ is uniformly bounded. This of course
implies that 
$(ix)^\alpha Y_1(ix) \to 0$ as $x \downarrow 0$ 
which in view of \eqref{defywis} requires that 
$g_\alpha(Y_2(ix)) \to 0$ as well. As
$g_\alpha: \R^+ \to \R^+$ is bounded away from zero
on compacts, we deduce that  
$Y_2(ix) \to \infty$ as $x \downarrow 0$, hence  
$h_\alpha(Y_2(i x)) \to 0$ and 
$$
\mu_\alpha^\gamma(\{0\})= 
\lim_{x\downarrow 0}h_\a (Y_1(ix))=1-\gamma\,,
$$
as claimed. 
Moreover, from the preceding   
$Y_1(ix) \to h_\alpha^{-1}(1-\gamma) := b \in \R^+$ as 
$x \downarrow 0$. Since $Y_1(z)$ is a
$\KK_\a$-valued continuous function of $z \in \C^+$, 
its cluster set $Cl(0)$ at the boundary point $z=0$ of $\C^+$ 
is a closed, connected subset of $\KK_\a$ 
(see \cite[Theorem 1.1]{collin}). Further, $Cl(0)$
contains $b \in \R^+$, so its boundary $\partial Cl(0)$ 
must intersect $[0,\infty)$.
We have seen that $Y_1(z)$ extends continuously 
on $(0,\infty)$ which due to the relation $Y_1(-\oo{z})=\oo{Y_1}(z)$
implies that it also extends continuously on
$(-\infty,0)$ with $Y_1(-t)=\oo{Y_1}(t)$ for all $t>0$. 
In particular, since the 
cluster set of $Y_1(t)$ for non-zero, real-valued $t \to 0$ contains 
$\partial Cl(0)$ (see \cite[Theorem 5.2.1]{collin}),   
necessarily the cluster set of $Y_1(\sqrt{t}\,)$
at the boundary point $t=0$ of $\R^+$ also 
intersects $[0,\infty)$. 

Using the bound $\sin(\zeta)/\zeta \ge 1-\zeta^2/6$,
we deduce 
from (\ref{deffal})
that if $\Im(h_\alpha(x+iy))=0$ 
for $y \ne 0$ then $y^2 \ge 6 h_\alpha'(x)/h_\alpha'''(x)$,
and direct calculation 
shows that this function of $x$ 
is positive and monotone non-decreasing.
Thus, with $Y_1(\sqrt{t}\,) \in {\mathcal K}_\alpha$ 
there exists $\delta>0$ such that 
if $\Im(h_\alpha(Y_1(\sqrt{t}\,)))=0$
then either $Y_1(\sqrt{t}\,) \ge 0$ is real-valued, 
or $|\Im(Y_1(\sqrt{t}\,))| \ge \delta$.
By \eqref{eq:den-wis}, 
the latter property applies whenever $t >0$ is such that 
$\rho_\alpha^\gamma(t)=0$. Moreover, by the
continuity of $Y_1(\cdot)$ on $(0,\infty)$,
%
if the density $\rho_\alpha^\gamma$ vanishes on an open 
interval $\Ib$, then either $Y_1(\sqrt{t}\,) \ge 0$ for 
all $t \in \Ib$ or  
$\inf_{t \in \Ib} |\Im(Y_1(\sqrt{t}\,))| \ge \delta$. 
For $\Ib=(0,\epsilon)$ we have already seen that 
the cluster set of $Y_1(\sqrt{t}\,)$ as $t \downarrow 0$
intersects $[0,\infty)$,
so necessarily $Y_1(\sqrt{t}\,) \in [0,\infty)$ 
for all $t \in \Ib$.  
Since \eqref{defywis} extends to 
$z \in \R^+$ and $g_\alpha(\R) \subseteq \R^+$ this in turn 
implies that $Y_2(\sqrt{t}\,) = i^\alpha r(t)$ for 
some continuous function $r:\Ib \mapsto \R^+$ 
such that $r(t) \to \infty$ as $t \downarrow 0$. The 
entire function $f_{\alpha,\theta} (z):=\frac{1}{2i} [h_\alpha(e^{i \theta} z)
-h_\alpha(e^{-i \theta} z) ]$ is then by \eqref{defGwis} such that
$$
f_{\alpha,\theta} (r(t))
=\Im\big(h_\alpha(e^{i \theta} r(t))\big)
=\Im\big(h_\alpha(Y_2(\sqrt{t}\,))\big)=0
$$
for $\theta=\pi \alpha/2$ and all $t \in \Ib$, which 
with $f'_{\alpha,\theta} (0)=\sin(\theta) h_\alpha'(0) \ne 0$  
contradicts the identity theorem.
We thus conclude that $\rho_\a^\gamma$ 
does 
not vanish on any non-empty interval $(0,\epsilon)$. 

\subsection{Properties of $\mu_\alpha$}
$~$

\noindent 
{\bf Proof of Proposition \ref{gammaone}.}
Taking $\s \equiv 1$ we deduce from Theorem 
\ref{theo-limitpoint-uniq-amir} that $Y(z)$ of \eqref{eq:BAG6}
is in $\KK_\a$ hence uniformly bounded on 
$\C^+ \setminus \{z: |z| < \delta \}$.
Similarly to the argument of Section \ref{sec:anal-lim}, if 
$y \in Cl (t)$ at $t>0$ 
then $y \in \KK_\a$ and 
$F(t,y):=t^\a y - C_\a g_\a(y)=0$,
so from the analyticity of $y \mapsto F(t,y)$ and uniform boundedness
of $g_\a(\cdot)$ on $\KK_\a$ we deduce 
by the identity theorem 
that $Y(z)$ extends continuously to a function $Y(t)$ on $(0,\infty)$.
%
Moreover, $t \mapsto Y(t)$ is real-analytic on $(0,\infty)$ 
outside the set 
of those $t>0$ where both $\partial_y F(t,y)=0$ and $F(t,y)=0$ 
at $y=Y(t)$. The latter set is clearly contained in the 
set $\DD^+_\a$ of $t>0$ such that   
$t^\a = C_\a g'_\a(y)>0$ for some $y \in \KK_\a$ at which 
$y g_\a'(y) - g_\a(y)=0$. Note that  
the set $\DD^+_\a$ is discrete since 
$y g_\a'(y) - g_\a(y)$ is an entire function of $y$. Further,  
$\DD^+_\al$ is a bounded set (by the uniform boundedness
of $g'_\a(\cdot)$ on $\KK_\a$, see Lemma \ref{boundg}). 
Consequently, $\DD^+_\a$ is a finite set. 
We already saw that $Y(-\oo{z})=\oo{Y}(z)$ for all
$z \in \C^+$, so $Y(-\oo{z})$ extends continuously 
to $Y(-t)=\oo{Y}(t)$ for any $t>0$ at which $Y(\cdot)$ extends 
continuously. We thus deduce that the exceptional set where 
$t \mapsto Y(t)$ may be non-analytic is contained in the finite
set $\{0,\pm t: t \in \DD^+_\al\}$, as claimed.
With $h_\a$ an entire function, it then follows 
that $G_\a(\cdot)$ extends continuously
to $\R \setminus \{0\}$ with the formula \eqref{rhoalpha}
for the symmetric density $\rho_\a(t)$ on $\R \setminus \{0\}$
that is real-analytic outside $\DD_\a$ (to verify 
the right-most expression in \eqref{rhoalpha} note that 
$h_\a(Y(z))=1-\frac{\alpha}{2 C_\a} z^\a Y(z)^2$
by \eqref{deffal} and \eqref{eq:BAG6}). 

If the symmetric density $\rho_\a$ vanishes on 
an open interval, then it also vanishes on some 
open interval $\Ib \subseteq \R^+$ where the 
continuous function $t\mapsto Y(t)$ is the 
limit of $Y(z)$ as $\arg(z) \downarrow 0$, hence 
$\arg(Y(t)) \in [0,\frac{\alpha \pi}{2}]$
(see \eqref{eq:amir-arg-id}). Further, the  
right-most expression in \eqref{rhoalpha} tells
us that $\sin(2\arg(Y(t))-\frac{\alpha\pi}{2})=0$ for 
all $t \in \Ib$, so necessarily $Y(t) = e^{i \theta} r(t)$
for $\theta = \frac{\alpha \pi}{4}$ and 
the continuous $r: \Ib \mapsto [0,\infty)$.
Since \eqref{eq:BAG6} extends to $t \in \Ib$ and
$g_\a(0) \ne 0$, we see that $Y(t) \ne 0$ is 
injective on $\Ib$, 
so $r(\Ib)$ contains an accumulation point. 
Finally, as argued at the end of 
Section \ref{sec:anal-lim}, from \eqref{rhoalpha} we also 
have that $f_{\a,\theta}(r(t))=\Im(h_\a(Y(t)))=0$
for the entire function $f_{\a,\theta}(\cdot)$ and
all $t \in \Ib$, yielding a contradiction. Consequently,
the density $\rho_\a$ does not vanish on any open interval,
as claimed.

It remains to show that $\mu_\a$ has a uniformly bounded density.
We get this by proving the stronger statement that $G_\a(z)$ is uniformly 
bounded on the connected set $\C^+_*:=\C^+ \cup \R^+$.
To this end, let $Cl_*(0)$ denote the cluster set 
of the continuous function $Y(z)$ at the boundary point $z=0$ of $\C^+_*$.
If $y \in \C$ is in $Cl_*(0)$ then there exists $z_n \in \C_*^+$ 
such that $z_n \to 0$ and $Y(z_n) \to y$, hence 
$g_\a(y)=0$ by \eqref{eq:BAG6}. Whereas $Cl_*(0)$ is a 
closed connected subset of $\C \cup \{\infty\}$ 
(by \cite[Theorem 1.1]{collin}), the set 
of zeros of the entire function $g_\a(\cdot)$ is discrete,
so necessarily $Cl_*(0)$ is a single point.
Taking $z=ix$, $x>0$ we have
that $Y(ix) \in \R^+$, hence $Y(ix) \to \infty$ by \eqref{eq:BAG6}
and the boundedness of $g_\a(\R^+)$, from which we deduce that
$Cl_*(0)=\{\infty\}$. Considering \eqref{eq:serban-new-bd} for 
$\beta=2$ we note that 
$|h_\a(y)| \le c_0 h_\a(\xi |y|)$ for some $\xi=\xi(\alpha)>0$, 
$c_0=c_0(\alpha)$ finite and all $y \in \KK_\a$. In particular, 
for $z \in \C^+_*$ such that $|z| \to 0$ we 
already know that 
$Y(z) \in \KK_\a$ and $|Y(z)| \to \infty$, hence 
by the preceding bound and the decay to zero of 
$h_\a(r)$ as $r \in \R^+$ goes to infinity, we have
that $h_\a(Y(z)) \to 0$. That is,
$Y(z) g_\a(Y(z)) \to 2/\alpha$ (see \eqref{deffal}). 
Next, observing that 
$h_\a(r) \le c_1 r^{-2/\alpha}$ for some 
positive, finite $c_1$ 
and all $r \in \R^+$,
we deduce from \eqref{new-amir-eq3} and \eqref{eq:BAG6} that 
for some finite constants $c_i=c_i(\alpha)$  
and all $z \in \C^+_*$,
\begin{eqnarray}\label{eq:new-bd-g}
|G_\a(z)| &=& |z|^{-1} |h_\a(Y(z))| \le c_0 |z|^{-1} h_\a(\xi|Y(z)|)
\nonumber \\
&\le& c_2 (|z^\alpha Y(z)^2|)^{-1/\alpha} = c_3 |Y(z) g_\a(Y(z))|^{-1/\alpha}
\,.
\end{eqnarray}
For any $\delta>0$ 
we have the uniform boundedness of $G_\a(z)$ on $\C_*^+ \cap \Ball
(0,\delta)^c$ 
(from the uniform boundedness of $h_\a$ on $\KK_\a$). 
Further, 
for $z \in \C^+_*$ converging to zero 
the right side of \eqref{eq:new-bd-g} remains bounded 
(by $c_3 (2/\alpha)^{-1/\alpha}$), 
hence $G_\a(z)$ is uniformly
bounded on $\C^+_*$, as stated.

\begin{rmk} We saw that $Y(ix) \in \R^+$ and
$x^\a Y(ix)^2 = |C_\a| Y(ix) g_\a(Y(ix)) \to 2 |C_\a|/\alpha$
as $x \downarrow 0$. With $\zeta=\sqrt{2|C_\a|/\a}$, it then
follows by dominated convergence that 
$\pi^{-1} x^{-1} h_\a(Y(ix)) \to 
\pi^{-1} \int_0^\infty \exp(-\zeta u^{\alpha/2}) du$
finite and positive. This is of course the value of 
$\rho_\a(0)$, provided $\rho_\a$ is continuous at $t=0$. 
\end{rmk}

\begin{lem}\label{alice-cont}
The measures $\mu_\a$ converge weakly to $\mu_2$ when $\a \uparrow 2$.
\end{lem}
\begin{proof} Applying the 
method of moments, 
as developed
by Zakharevich \cite{zakh}, it is shown in 
\cite[Theorem 1.8]{BAG6} that for any $B<\infty$ fixed,
$\E [\mun_{\bA_N^B}]$ converges
to some non-random $\mu^B_\a$ as $N \to \infty$
(for instance, 
when  $x_{ij}$ are stable variables of index $\alpha$).
Examining the dependence of $C(B)$ of \cite[equation (13)]{BAG6}
on $\a$, we see that \eqref{trunc1} applies for 
some $\d(\e,B)>0$, all $B>B(\e,\a_0)$ and any $\a \in (\a_0,2)$.
For such $B$ and $\a$ we thus have, in view of the 
almost sure convergence of $\mun_{\bA_N}$ to
$\mu_\a$, that 
$\P(d_1(\mu_\a,\mun_{\bA_N^B}) \ge 3\e) \to 0$ as
$N \to \infty$, from which we deduce by the 
boundedness and convexity of $d_1$ that
$$
d_1(\mu_\a,\mu_\a^B) = \lim_{N \to \infty}
d_1(\mu_\a,\E[\mun_{\bA_N^B}]) \le 
\limsup_{N \to \infty} \E \big[ d_1(\mu_\a,\mun_{\bA_N^B})\big] \le 
3\e \,.
$$
Fixing $B<\infty$ it further 
follows from \cite[Lemmas 9.1 and 9.2]{BAG6}
that $\mu^B_\alpha$ converges weakly to the
semi-circle $\mu_2$ when $\alpha \to 2$. Hence, fixing  
$\a_0>0$, $\e>0$ and $B>B(\e,\a_0)$, by the triangle inequality 
$$
d_1(\mu_\alpha,\mu_2)\le 
d_1(\mu_\a,\mu^B_\a) +d_1(\mu_\alpha^B,\mu_2)
\le 3\e + d_1(\mu_\a^B,\mu_2) \to 3\e 
$$
as $\a \uparrow 2$. Taking $\e \downarrow 0$ we thus
conclude that $\mu_\a \to \mu_2$ when $\a \uparrow 2$.
\end{proof}

\section{Diagonal perturbation: 
Proof of Theorem \ref{weakeningD}}\label{sec:diag}

\subsection{The extension of Theorem \ref{theo-limitpoint-uniq-amir}}
\label{sec:ext-thm12}
We shall prove the convergence 
of the expected spectral measures $\E[\mun_{\bA_N + \bD_N}]$ and 
characterize their limit in case
$\sigma\in\PC$ is given as in Section \ref{sec:ind}
by \eqref{eq:sPC} for some $q \in \N$, 
$0=b_0<b_1<\cdots<b_{q}=1$
and $\sigma_{rs}=\sigma_{sr}$ with the 
corresponding 
random matrix $\bA_N=\bA_N^\sigma$ and the 
$N\ts N$ piecewise constant matrix $\bsigma^N$.
To this end, recall 
that $\bD_N$ is a diagonal $N \ts N$ matrix, whose entries 
$\{D_N(k,k), 1 \le k \le N \}$ are real valued,  
independent of the random variables $(x_{ij}, 1 \leq i \le j < \infty)$
and identically distributed, of law $\mu^{\bD}$ having a finite second 
moment. In view of the assumed finite second moment of $\mu^{\bD}$,  
the proof of \eqref{trunc2} and Lemma \ref{tight} also show  
that the sequences $(\E[\mun_{\bA_N+\bD_N}];N\in\N)$, 
$(\E[\mun_{\bA^B_N+\bD_N}]; N\in\N)$ and
$(\E[\mun_{\bA_N^\kappa+\bD_N}]); N\in\N)$ are 
tight for the topology of weak convergence on $\Pa(\R)$,
and that $(\E[\mun_{\bA_N^\kappa+\bD_N}]); N\in\N)$ has the same  
set of limit points as $(\E[\mun_{\bA_N+\bD_N}];N\in\N)$.
Setting now $\bG_N(z)=(z\bI_N-\bD_N-\bA_N)^{-1}$ we define for $z\in\C^+$
the probability measures 
$L_N^z$ 
and $L_{N,r}^z$ on $\C$ as in \eqref{eq:lnz} and \eqref{eq:lnzr},
with $\bG_N^\kappa(z)$ and $L_{N,r}^{z,\kappa}$ denoting again 
the corresponding objects when $\bA_N$ is replaced by $\bA_N^\kappa$.

For $0< \kappa < \frac{1}{2(2-\alpha)}$
any $1\le r\le q$ and bounded Lipschitz function $f$ we
then have similarly to Lemma \ref{approxss} that as $N \to \infty$
\begin{equation}\label{eq:approxes-ext}
\Big|
\E\big[L_{N,r}^{z,\kappa}(f)\big]-
\E\Big[f\Big( \big(z- D_N(0,0) - \sum_{k=1}^N
\widetilde A_N^\kappa([N b_r],k)^2 G_N^\kappa(z)_{kk}\big)^{-1}\Big)\Big]
\Big| \to 0 \,,
\end{equation}
where $\widetilde \bA_N^\kappa$ denotes an independent copy
of $\bA_N^\kappa$ which is also independent of $\bD_N$ while  
$D_N(0,0)$ of law $\mu^{\bD}$ is independent of all other variables. 
Indeed, focusing w.l.o.g. on $r=1$ and taking 
$\bar \bG^\kappa_{N+1}(z) = (z\bI_{N+1}-\bar \bD_{N+1}- \bar 
\bA_{N+1}^\kappa)^{-1}$ 
(with $\bar \bD_{N+1}$ denoting the diagonal matrix of entries
$D_N(k,k)$, $k=0,\ldots,N$), 
we get \eqref{eq:lbrn00} by the invariance of the law
of $\bar \bD_{N+1} + \bar \bA_{N+1}^\kappa$ to
symmetric permutations of its first $[N b_1]+1$ rows 
and columns.
Schur's complement formula then leads to the identity 
\eqref{equalitylemma1}
with $D_N(0,0)$ added to $\wt{A}^\kappa_N(0,0)$ on its right side. 
All eigenvalues (and diagonal terms) of $\bG_N^\kappa(z)$ 
are in the compact set $\Db (z)$, regardless of the value of $\bD_N$,
and the centered entries of $\wt{\bA}_N^\kappa$ are independent of 
both $\bG_N^\kappa(z)$ and $D_N(0,0)$. Thus, as in the 
proof of Lemma \ref{approxss} we can neglect both $\wt{A}_N^\kappa(0,0)$ and  
$\sum_{k\neq l} \wt{A}_N^\kappa(0,k) \wt{A}_N^\kappa(l,0) G_N^\kappa(z)_{kl}$ 
in \eqref{equalitylemma1} and get \eqref{approx2} except for
changing here $z$ to $z-D_N(0,0)$ in its right side. 
Equipped with the latter version of \eqref{approx2}, 
fixing $0< \kappa < \frac{1}{2(2-\alpha)}$ we arrive at 
\eqref{eq:approxes-ext} upon adapting \cite[Lemma 4.1]{BAG6} and 
its proof to our matrices $\bar \bG^\kappa_{N+1}$ 
and $\bG^\kappa_N$ (while taking there the corresponding matrices 
$\hat{\bG}^\kappa_N = (z\bI_{N+1} -\bar 
\bD_{N+1} - \hat{\bA}^\kappa_N)^{-1}$).
 
The concentration result of Lemma \ref{concentration} holds 
in the presence of the diagonal matrix $\bD_N$ 
of i.i.d. entries. Indeed, its proof is easily adapted to the
current setting by considering for $f$ continuously differentiable 
$L_{N,s}^{z,\kappa}(f) := 
F_{N}(D_N(l,l), A^\kappa_N({k,l}), 1 \le k\le l \le N)$, and 
noting that for $1 \leq l \leq N$, 
$$
\partial_{D(l,l)} F_N = \frac{1}{N} [\bG_N^\kappa(z) \bD_s(f')
\bG_N^\kappa(z)]_{ll}
\,.
$$
The spectral radius of $\bG_N^\kappa (z) \bD_s(f') \bG_N^\kappa (z)$ 
is again bounded by $\|f'\|_\infty /  |\Im(z) |^2$, so 
$\sup_l \|\partial_{D(l,l)}F_N\|_\infty 
\leq \|f\|_\rBL (N |\Im(z) |^2)^{-1}$. 
There are only $N$ such variables $\{D_N(l,l)\}$ 
to consider, each having the same finite second moment, 
so using the same martingale bound as in \eqref{danslebureau},
their total effect on 
$\E[ (F_N-\E[F_N])^2]$ 
is taken care off by enlarging the finite constant $c_0$.

Equipped with this concentration result and 
replacing Lemma \ref{approxss} with \eqref{eq:approxes-ext},
we follow the proof of Proposition \ref{prop-limitpoint}
to deduce that in our current setting, 
for $r\in\{1,\cdots,q\}$ 
and every bounded continuous function $f$ on $\Db (z)$,
\begin{equation}\label{eq:fidend}
\int f d\mu^z_r =\int f\Big( (z- \lambda - \sum_{s=1}^q
\sigma_{rs}^2  \Delta_s^{\frac{2}{\alpha}} x_s)^{-1}\Big) 
\prod_{s=1}^qdP^{\mu^z_s}(x_s) d\mu^{\bD}(\lambda) \,.
\end{equation}
Following the proof of Proposition \ref{projectionoflimitpoint}
we find that
this in turn implies that any subsequence of the functions  
$X_{N,r}(z) = \E[L_{N,r}^{z,\kappa}(x^{\alpha/2})]$ 
has at least one limit point $(X_{r}(z),1 \le r \le q)$ 
composed of analytic functions on $\C^+$ that are bounded
by $(\Im(z))^{-\alpha/2}$ and satisfy the 
following generalization of \eqref{cocottes}
$$
X_r(z)= C(\alpha)\int \int_0^\infty t^{-1} (it)^{\frac{\alpha}{2}}
e^{it(z-\lambda)} \exp\{-(it)^{\frac{\alpha}{2}} \wh{X}_r(z) \}
\, dt \, d \mu^{\bD}(\lambda)\,,
$$
for the analytic functions $\wh{X}_r :\C^+ \mapsto \wh{\KK}_\a$ 
of \eqref{eq:yrdef}.
 
We proceed to extend Proposition \ref{uniqueprop} 
to the setting of $\bA^\s_N+\bD_N$. Indeed, fixing 
$z \in \C^+$, 
upon applying per $\lambda \in \reals$  the identity \eqref{eq:gziden} 
for $\beta=\alpha$, $y = (\lambda-z)^{-\alpha/2} \wh{X}_r(z)$ 
and with $z-\lambda \in \C^+$ replacing $z$, we see that 
the preceding generalization of \eqref{cocottes} 
is equivalent to 
$$
X_r(z)= C(\alpha)\int (\lambda-z)^{-\frac{\alpha}{2}}
g_{\a,\a} ( (\lambda-z)^{-\frac{\a}{2}} \wh{X}_r (z) ) 
\, d \mu^{\bD}(\lambda)\,.
$$
By \eqref{eq:yrdef} we thus deduce that
$(\wh{X}_r(z), 
1 \leq r \leq q)$ satisfy \eqref{systemeqhX}. 
Namely, it is a solution of $\uu{\wh{x}}=\uu{F}_{z}(\uu{\wh{x}})$
composed of analytic functions from $\C^+$ to $\wh{\KK}_\alpha$,
where $\uu{F}_z(\cdot)=(F_{z,r}(\cdot), 1 \leq r \leq q)$ and 
$$
F_{z,r}(\uu{\wh{x}}) 
:= 
\oo{C}_\a \sum_{s=1}^q 
\wh{a}_{rs} \int (\lambda-z)^{-\frac{\alpha}{2}}
g_{\alpha} \big((\lambda-z)^{-\frac{\alpha}{2}} \wh{x}_s\big) 
d\mu^{\bD}(\lambda) \,, 
$$
for $\wh{a}_{rs}=|\sigma_{rs}|^\alpha \D_s$.
Note that if $\wh{x}_s \in \wh{\KK}_\alpha$ then 
$(\lambda-z)^{-\frac{\alpha}{2}} \wh{x}_s$
is in $\KK_\alpha$ so such solutions must have  
$|\wh{x}_r| \leq c (\Im(z))^{-\frac{\alpha}{2}}$ 
for
 $c := |C_\a| 
\|g_\alpha\|_{\KK_\alpha} \max_{r} \sum_{s=1}^q |\wh{a}_{rs}|$ 
finite and all $z \in \C^+$. Consequently, if $\Im(z) \geq 1$
then 
$\max_r |\wh{x}_r | \leq c$. 
Thus, for such $z$, any  $1 \leq r \leq q$ and any 
two fixed points $\uu{\wh{x}}$ and $\uu{\wh{y}}$ of 
$\uu{F}_z(\cdot)$ in $(\wh{\KK}_\alpha)^q$, 
$$
|F_{z,r}(\uu{\wh{x}}) - F_{z,r}(\uu{\wh{y}})| \leq
\max_{r,s} \{|\wh{a}_{rs}| \} \|g_\alpha\|_c (\Im (z))^{-\alpha} 
\| \uu{\wh{x}} - \uu{\wh{y}} \|_1
$$
(where $\|g_\alpha\|_c$ and $\|g_\alpha\|_{\KK_\alpha}$ are as in 
the proof of Theorem \ref{weakening} and
$\|\uu{\wh{x}}\|_1 := \sum_{s=1}^q |\wh{x}_s|$). 
Thus, for some $k_0$ finite, if $\Im(z) \ge k_0$ then 
$\|\uu{F}_{z}(\uu{\wh{x}}) - \uu{F}_{z}(\uu{\wh{y}})\|_1 \leq
\frac{1}{2} \| \uu{\wh{x}} - \uu{\wh{y}} \|_1$ resulting with uniqueness of 
the fixed point of $\uu{F}_z(\cdot)$ in $(\wh{\KK}_\alpha)^q$.
This in turn implies the stated uniqueness of such fixed point composed
of analytic functions $z \mapsto \wh{x}_s$ from $\C^+$ to $\wh{\KK}_\alpha$.

To complete the proof of Theorem \ref{weakeningD} in case
$\s \in \PC$, we adapt our 
proof of Theorem \ref{theo-limitpoint-uniq-amir}, where instead of
\eqref{eq:newam2}, combining \eqref{eq:fidend}
for $f(x)=x$ 
with \eqref{cocott}, here the limit points $\mu^z_s$ of  
$(\E[L_{N,s}^{z,\kappa}], 1 \leq s \leq q)$
are such that for each $r \in \{1,\ldots, q\}$,
\begin{equation}\label{eq:newam2d}
\int xd\mu^z_r(x) = -i \int d\mu^{\bD}(\lambda) 
\int_0^\infty e^{it(z-\lambda)} \exp\{- (it)^{\frac{\alpha}{2}}
\wh{X}_r(z)  \} dt \,.
\end{equation}
In particular, since $\int x d\mu^z_r$ is uniquely determined 
by $\wh{X}_r(z)$ 
we deduce that the sequence  
$\E[L_{N,r}^{z,\kappa}(x)]$ converges as $N \to \infty$
to the right side of \eqref{eq:newam2d}. So, with  
$$
\E[\int \frac{1}{z-x} d\mun_{\bA_N^{\kappa}+\bD_N}(x)]
=\sum_{s=1}^q\De_{N,s}\E[L_{N,s}^{z,\kappa}(x)]
$$
for any $z \in \C^+$, it follows 
that 
$$
\int \frac{1}{z-x} d\mu^{\sigma,\bD}(x) =-i \sum_{s=1}^q\De_s 
\int d\mu^{\bD}(\lambda) 
\int_0^\infty e^{it(z-\lambda)} \exp\{- (it)^{\frac{\alpha}{2}}
\wh{X}_s(z)  \} dt \,,
$$
for any 
limit point $\mu^{\sigma,\bD}$ of $\E[\mun_{\bA_N^{\kappa}+\bD_N}]$.
With the Cauchy-Stieltjes transform $G^{\bD}_{\alpha,\sigma}$ 
of $\mu^{\sigma,\bD} \in \Pa(\R)$ 
uniquely determined, we deduce that 
$\E[\mun_{\bA_N^{\kappa}+\bD_N}]$ converges to 
$\mu^{\sigma,\bD}$, hence so does $\E[\mun_{\bA_N+\bD_N}]$.
Finally, for $z \in \C^+$ we arrive at the formula 
\begin{equation}\label{eqGD} 
G^{\bD}_{\alpha,\sigma}(z)
= 
\int \frac{1}{z-\lambda}  \sum_{s=1}^q \D_s 
h_\alpha((\lambda-z)^{-\frac{\alpha}{2}} \wh{X}_s(z)) 
\,
d\mu^{\bD}(\lambda) \,,
\end{equation}
by applying \eqref{eq:gziden} with $\beta=2$,  
$y=(\lambda-z)^{-\frac{\alpha}{2}} \wh{X}_s(z)$ 
and $z-\lambda$ instead of $z$. 

\subsection{The extension of Theorem \ref{weakening}}
Setting $\sigma \in \FF_\alpha$
we adapt the proof of Theorem \ref{weakening} to the current 
setting. Indeed, using the same approximating sequence 
$\s_p\in\PC$ of $\s \in \FF_\alpha$ as in the 
proof of Theorem \ref{weakening}, we have shown already that
\eqref{systeqlimD} holds for each of the piecewise constant
functions $\wh{X}^{\sigma_p}_. (z) : (0,1] \to \wh{\KK}_\alpha$,
$p \in \N$, where 
$$
\wh{X}^{\sigma_p}_x(z)=\wh{X}_s(z) \mbox{ for } x\in (b_{s-1}^p,b_s^p]
\mbox{ and } s=1,\ldots,q(\s_p) \,,
$$
and $\wh{X}_s (z) \in \wh{\KK}_\alpha$ are the unique
collections of (analytic) functions of $z \in \C^+$ 
we have constructed in Section \ref{sec:ext-thm12}.

Similarly to the proof of Theorem \ref{weakening}, 
we get the existence of 
a bounded measurable solution 
$\wh{X}^{\sigma}_. (z) : (0,1] \mapsto \wh{\KK}_\alpha$
of \eqref{systeqlimD} whenever $\Im(z) \ge R = R(\sigma)$ 
by showing that for such $z$ 
the fixed points $(\wh{X}^{\sigma_p}_. (z), p \in \N)$ of the mappings 
$$
F_z(\s,\wh{X}) := \oo{C}_\alpha \int_0^1 |\sigma(\cdot,v)|^\alpha
\int (\lambda-z)^{-\frac{\alpha}{2}} g_{\alpha} \big( 
(\lambda-z)^{-\frac{\a}{2}} \wh{X}_v \big) d\mu^{\bD}(\lambda) \, dv \,,
$$
at $\s=\s_p$ form a Cauchy sequence in $L^\infty((0,1])$.
To this end, recall that $\|g_\alpha\|_{\KK_\alpha}$ is finite (by
Lemma \ref{boundg}), so fixing $\epsilon \in (0,1)$ and 
bounding the 
$L^\infty$-norm of $F_z(\s,\wh{X})$ for $\wh{X} \in \wh{\KK}_\alpha$ 
we deduce that $\|\wh{X}^{\sigma_p}_.\|_\infty \le r_\s$ of \eqref{unifbY}
for all $p \in \N$, whenever $\Im(z) \geq \epsilon$. It is easy to verify 
that for such $z$ our mapping $F_z(\cdot,\cdot)$ 
satisfies the inequality \eqref{eq:Fbd} except for 
replacing there $|z|^{-\a}$ by $(\Im(z))^{-\a/2}$. Consequently, 
with $\wh{X}^{\sigma_p}_.$ fixed points of this mapping, 
our uniform bound on $\|\wh{X}^{\sigma_p}_.\|_\infty$ 
implies that 
\begin{eqnarray*}
\| \wh{X}^{\sigma_q}_.-\wh{X}^{\sigma_p}_.\|_\infty \le
(\Im(z))^{-\alpha/2} \|g_\alpha\|_{r_\s} 
\Big[ \| \, |\s_q|^\alpha - |\s_p|^\alpha\| + 2 k_\s   
\| \wh{X}^{\sigma_q}_.-\wh{X}^{\sigma_p}_.\|_\infty \Big] \,,
\end{eqnarray*}
for any $p,q \in \N$ and $\Im(z) \geq \epsilon$.
Thus, setting $R \ge \epsilon$ such that 
$R^{-\alpha/2} \|g_\alpha\|_{r_\s} k_\s \leq 1/3$, 
we conclude in view of \eqref{domaine} that 
$(\wh{X}^{\sigma_p}_.,p\in \N)$ is a Cauchy sequence
in $L^\infty(0,1];\C)$ whenever $z$ is in 
$\C^+_R := \{ z : \Im(z) > R \}$.
As in the proof of Theorem \ref{weakening},
the $L^\infty$-norm of its limit $\wh{X}^{\s}_.$ 
is at most $r_\s$ so by \eqref{domaine} and
the modified inequality \eqref{eq:Fbd} 
$\wh{X}^{\s}_. (z)$ must be a fixed point of 
$F_z(\s,\cdot)$. Further, equipped with the latter 
inequality, the uniqueness (almost everywhere) 
of such a solution to \eqref{systeqlimD} is obtained
by a re-run of the relevant argument from the
proof of Theorem \ref{weakening}.
We have seen that the holomorphic mappings 
$\wh{X}^{\s_p}$ from $\C^+$ to the closed subset
$\Fb :=L^\infty((0,1];\wh{\KK}_\alpha)$ of 
$L^\infty((0,1];\C)$ are locally uniformly bounded.
Hence, their $L^\infty$-convergence to $\wh{X}^\s$ 
extends by Vitali's convergence theorem 
from the non-empty open subset $\C^+_R$ 
to all of $\C^+$, with $\wh{X}^\s: \C^+ \mapsto \Fb$ 
an analytic mapping which is uniquely determined 
by the uniqueness of the solution in $\Fb$ of
\eqref{systeqlimD} for each $z \in \C^+_R$ 
(and the identity theorem).

Next, with the same proof as in Proposition \ref{prop-dense} 
we have from the $L^2_\star$-convergence of $\s_p$ to $\s$ that 
$G^{\bD}_{\alpha,\s_p}(z) \to G^{\bD}_{\a,\s}(z)$ as $p \to \infty$, 
for each $z \in \C^+$. If $z \in \C^+_R$ then also
$\| \wh{X}^{\s_p}_. - \wh{X}^\s_.\|_\infty \to 0$. As  
the identity \eqref{eqGlimD} holds for  
$\s=\s_p \in \PC$, $p \in \N$ (being then
merely the formula \eqref{eqGD}),
taking $p \to \infty$ we deduce by dominated convergence  
that \eqref{eqGlimD} holds for $\s \in \FF_\alpha$
and $z \in \C^+_R$. 
For all $x \in (0,1]$, $\lambda \in \R$ and $z \in \C^+$
the argument $(\lambda-z)^{-\alpha/2} \wh{X}_x^\s(z)$ of 
the entire function $h_\alpha$ is in the set $\KK_\a$ 
where $h_\alpha$ and its derivatives are 
uniformly bounded. 
Further, for such $\lambda$ the mapping  
$z \mapsto (\lambda-z)^{-\alpha/2} \wh{X}^{\s} (z)$  
from $\C^+$ to $L^\infty((0,1];\C)$ is analytic,
out of which 
one can verify that the right side of 
\eqref{eqGlimD} is analytic on $\C^+$. 
With $G^{\bD}_{\alpha,\s}$ also analytic on $\C^+$ 
the validity of \eqref{eqGlimD} extends 
from $\C^+_R$ to $\C^+$
(by the identity theorem).

\medskip
{\bf Acknowledgment}
We thank Ofer Zeitouni for telling us about the 
interpolation approach to convergence almost surely 
(as done here in Lemma \ref{lem:interp}), 
and Gerard Ben Arous for proposing that we  
investigate the gap at zero for the support of 
the measure $\mu_\a^\gamma$ and 
suggesting the argument for the
continuity of $\alpha\mapsto \mu_\alpha$ at 
$\alpha=2$. We are also grateful to 
Alexey Glutsyuk and 
Jun Li for their help in forming  
Proposition \ref{prop:analytic} and to the anonymous referee 
for pointing out the problem in using an $L^2$-norm instead 
of the $L^2_\star$ semi-norm.

\bibliography{wishart}
\bibliographystyle{acm}

\end{document}